\documentclass[12pt]{article}
\usepackage{amsmath,enumerate,amsfonts,amssymb,color,graphicx,amsthm}
\usepackage{mathrsfs}

\usepackage{hyperref}
\usepackage{todonotes}
\usepackage[normalem]{ulem}

\usepackage{cite}

\def\RR{{\mathbb R}}

\def\SSphere{{\mathbb S}}

\def\Ric{{\rm Ric}}

\newcounter{marnote}

\newtheorem{theorem}{Theorem}[section]
\newtheorem{proposition}[theorem]{Proposition}
\newtheorem{corollary}[theorem]{Corollary}
\newtheorem{lemma}[theorem]{Lemma}

\newtheorem{remark}[theorem]{Remark}

\def\deg{{\rm deg}\,}
\def\vareps{\varepsilon}

\def\mcL{{\mycal L}}

\def\tr{\textrm{tr}}

\DeclareFontFamily{OT1}{rsfs}{}
\DeclareFontShape{OT1}{rsfs}{m}{n}{ <-7> rsfs5 <7-10> rsfs7 <10-> rsfs10}{}
\DeclareMathAlphabet{\mycal}{OT1}{rsfs}{m}{n}

\def\mcO{{\mycal O}}
\def\mcS{{\mycal S}}
\def\mcN{{\mycal N}}

\def\roundg{{g_0}}

\def\ringg{{\mathring{g}}}

\newcommand\NewtonT[1]{{\stackrel{(#1)}{\mathop{T}}}}

\begin{document}
\title{On the $\sigma_{k}$-Nirenberg problem}
\author{YanYan Li \thanks{Department of Mathematics, Rutgers University, Hill Center, Busch Campus, 110 Frelinghuysen Road, Piscataway, NJ 08854, USA. Email: yyli@math.rutgers.edu.}~\thanks{Partially supported by NSF Grants DMS-1501004, DMS-2000261, and Simons Fellows Award 677077.}~, Luc Nguyen \thanks{Mathematical Institute and St Edmund Hall, University of Oxford, Andrew Wiles Building, Radcliffe Observatory Quarter, Woodstock Road, Oxford OX2 6GG, UK. Email: luc.nguyen@maths.ox.ac.uk.}~ and Bo Wang \thanks{School of Mathematics and Statistics, Beijing Institute of Technology, No. 5, Zhongguancunnan Street, Haidian District, Beijing 100081, China. Email: wangbo89630@bit.edu.cn.}~\thanks{Partially supported by NNSF (11701027) and Beijing Institute of Technology Research Fund Program for Young Scholars.}}

\date{}

\maketitle

\begin{abstract}
We consider the problem of prescribing the $\sigma_k$-curvature on the standard sphere $\SSphere^n$ with $n \geq 3$. We prove existence and compactness theorems when $k \geq n/2$. This extends an earlier result of Chang, Han and Yang for $n = 4$ and $k = 2$.
\end{abstract}

\tableofcontents
\section{Introduction}

 The Nirenberg problem, raised by Nirenberg in the years 1969--1970, asks to identify functions $K$ on the two-sphere $\SSphere^2$ for which there exists a metric $g$ on $\SSphere^2$ conformal to the standard metric $\roundg$ such that $K$ is the Gaussian curvature of $g$. Naturally, this problem extends to higher dimensions with the Gaussian curvature replaced by the scalar curvature.

The Nirenberg problem has been studied extensively since 1970s and it would be impossible to mention here all works in this area. An important aspect most directly related to the present work is the fine analysis of blow-up (approximate) solutions or the compactness of the solution set. In dimensions $n \leq 3$, this was studied in Chang and Yang \cite{ChangYang87-Acta, ChangYang88-JDG}, Bahri and Coron \cite{BahriCoron91-JFA}, Chang, Gursky, Yang \cite{ChangGurskyYang93-CVPDE}, Han \cite{Han90-Duke}, Zhang \cite{Zhang90-Thesis}, Schoen and Zhang \cite{SchoenZhang96-CVPDE}. In higher dimensions, this was studied in Li \cite{Li95-JDE, Li96-CPAM} and Chen and Lin \cite{ChenLin99-CPDE}. We note that there is a distinctive difference when $n \leq 3$ and $n \geq 4$. In dimensions $n \leq 3$, every sequence of solutions cannot blow up at more than one point. On the other hand, in dimensions $n \geq 4$, a sequence of solutions can blow up at multiple points \cite{Li96-CPAM}, while in dimensions $n \geq 7$, a sequence of solutions can have unbounded energy \cite{ChenLin99-CPDE}. For further studies on the Nirenberg problem, see Chen and Ding \cite{ChenDing87-TrAMS}, Chang and Liu \cite{ChangLiu93-IJM}, Aubin and Bahri \cite{AubinBahri97-JMPA1, AubinBahri97-JMPA2}, Malchiodi and Mayer \cite{MalchiodiMayer19-arxiv, MalchiodiMayer20-JDE}, and the references therein.

In this paper we study a fully nonlinear version of the Nirenberg problem on $\SSphere^n$ with $n \geq 3$. For a metric $g$ on $\SSphere^{n}$, let $\Ric_g$, $R_g$ and $A_{g}$ denote the Ricci curvature, the scalar curvature and the Schouten tensor of $g$,
\begin{equation*}
A_{g}=\frac{1}{n-2}\left(\Ric_{g}-\frac{R_{g}}{2(n-1)}g\right).
\end{equation*}
Let $\lambda(A_{g})=(\lambda_{1},\ldots,\lambda_{n})$ denote the eigenvalues of $A_{g}$ with respect to $g$. For an integer $1\leq k\leq n$, let $\sigma_{k}(\lambda)=\sum\limits_{1\leq i_{1}<\cdots<i_{k}\leq n}\lambda_{i_{1}}\cdots\lambda_{i_{k}}$, $\lambda=(\lambda_{1},\ldots,\lambda_{n})\in\RR^{n}$, denote the $k$-th elementary symmetric function, and let $\Gamma_{k}$ denote the connected component of $\{\lambda\in\RR^{n}:\sigma_{k}(\lambda)>0\}$ containing the positive cone $\{\lambda\in\RR^{n}:\lambda_{1},\ldots,\lambda_{n}>0\}$. The function $\sigma_k(\lambda(A_g))$ will be referred to as the $\sigma_k$-Schouten-curvature, or simply the $\sigma_k$-curvature, of $g$.

Let $\roundg$ be the round metric on $\SSphere^n$ and $[\roundg]$ be the set of metrics conformal to $\roundg$. When $g \in [\roundg]$, say $g = g_{v} := v^{\frac{4}{n-2}}\roundg$ for some positive function $v$, we have
\begin{equation*}
A_{g_{v}}= \frac{1}{2} g_0 -\frac{2}{n-2}v^{-1}\nabla^{2}_{\roundg}v+\frac{2n}{(n-2)^{2}}v^{-2}dv\otimes  d v-\frac{2}{(n-2)^{2}}v^{-2}|dv|^{2}_{\roundg}\roundg.
\end{equation*}
We are interested in identifying $K$ such that the problem
\begin{equation}\label{yq}
\sigma_{k}(\lambda(A_{g_v}))=K(x),\quad\lambda(A_{g_v})\in\Gamma_{k},\quad\quad\mbox{on }\SSphere^{n}
\end{equation}
has a positive solution.

Problem \eqref{yq} can also be asked on manifolds other than the sphere. In such context, it is usually referred to as the $\sigma_k$-Yamabe-type problem and has been the subject of many studies in the past 20 years. For an account of works in this area, we refer the readers to the following recent papers \cite{AbantoEspinar19-AJM, BarbosaCavalcanteMarcosEspinar19-CPAM, BoSheng19-JGA, Case19-PJM, CaseWang18-AiM, CaseWang20-JMS, FangWei19-CVPDE, FangWei20-Arxiv, GonLiNg, GurskyStreets18-GT, HanLiLi21-CPAM, He21-MAnn, JiangTrudinger-ObBVP-1,JiangTrudinger-ObBVP-2,JiangTrudinger-ObBVP-3 , LiNgLipNeg, LiNgGreen, LiWang19-ATA, Santos17-Indiana, Sui, Trudinger-TianVolume} and references therein.

In this paper, we focus on the case $k \geq n/2$. We assume that $K$ satisfies the non-degeneracy condition:
\begin{equation}\label{nondegeneracy}
|\nabla_{g_0} K|_{g_0} + |\Delta_{g_0} K| > 0 \text{ on }\SSphere^n.
\end{equation}

For $K$ satisfying \eqref{nondegeneracy}, the set of critical points of $K$ can be split as $\textrm{Crit}_+(K)  \cup \textrm{Crit}_-(K) $ where
\begin{align*}
\textrm{Crit}_+(K) 
	&= \{x \in \SSphere^n | \nabla_{\roundg} K(x) = 0, \Delta_{\roundg} K(x) > 0\},\\
\textrm{Crit}_-(K) 
	&= \{x \in \SSphere^n | \nabla_{\roundg} K(x) = 0, \Delta_{\roundg} K(x) < 0\}.
\end{align*}
It is well-known that if $O$ is an open subset of $\SSphere^n$ such that $\nabla K \neq 0$ on $\partial O$, then the degree $\deg(\nabla K, O,0)$ is well-defined. Set 
\[
\deg(\nabla K, \textrm{Crit}_-(K)) := \deg(\nabla K, O, 0)
\]
where $O$ is any open subset of $\SSphere^n$ containing $\textrm{Crit}_-(K)$ and disjoint from $\textrm{Crit}_+(K)$. Due to \eqref{nondegeneracy}, this is well-defined. Furthermore, the map $K \mapsto \deg(\nabla K, \textrm{Crit}_-(K))$ is continuous and integer-valued on the set of $C^2$ functions satisfying \eqref{nondegeneracy}: If $K$ satisfies \eqref{nondegeneracy} and if $K_i \rightarrow K$ in $C^2$, then $\deg(\nabla K_i, \textrm{Crit}_-(K_i)) = \deg(\nabla K, \textrm{Crit}_-(K))$ for large $i$. When $\textrm{Crit}_-(K)$ consists of only isolated non-degenerate points,
\[
\deg(\nabla K, \textrm{Crit}_-(K))  = \sum_{x_0 \in \textrm{Crit}_-(K)} (-1)^{i(x_0)}
\]
where $i(x_0)$ is the number of negative eigenvalues of $\nabla_{\roundg}^2 K(x_0)$. (For an introduction to degree theories, see e.g. \cite{NirenbergLectNotes}.)

\begin{theorem}\label{main'}
Let $n\geq3$ and $n/2\leq k\leq n$. Suppose that $K \in C^2(\SSphere^{n})$ is a positive function satisfying the non-degeneracy condition \eqref{nondegeneracy}. Then there exists a positive constant $C_*$, depending only on $n$ and $K$, such that
\begin{equation}
 \|\ln v\|_{C^{2}(\SSphere^{n})} \leq C_* \text{ for all $C^2$ positive solutions $v$ of \eqref{yq}}.\label{uniform estimate}
\end{equation}
Furthermore, if $\deg(\nabla K, \mathrm{Crit}_-(K)) \neq (-1)^n$, then (\ref{yq}) admits a positive solution.
\end{theorem}

\begin{remark}
For $n=4$ and $k=2$, Theorem \ref{main'} was proved by Chang, Han and Yang  \cite{CHY11-CVPDE}.
\end{remark}

\begin{remark}
See Theorem \ref{thm9999} for more detailed dependence of $C_*$ on $K$.
\end{remark}

\begin{remark}
If $K \in C^{2,\alpha}(\SSphere^n)$, $0 < \alpha < 1$, and $\mcO$ is an open subset of $C^{4,\alpha}(\SSphere^n)$ which contains all positive solutions to \eqref{yq} and whose elements $v$ are positive and satisfy $\lambda(A_{g_v}) \in \Gamma_k$, then
\[
\deg(\sigma_{k}(\lambda(A_{g_v}))-K, \mcO, 0) = -1 + (-1)^n \deg(\nabla K, \textrm{Crit}_-(K)).
\]
Here the degree is as defined in \cite{Li89-CPDE}. See Theorem \ref{Thm:dff}.
\end{remark}

We note that the existence and compactness of solutions for the $\sigma_k$-Yamabe problem on compact manifolds not conformally equivalent to the round sphere was proved for $k = 2$ and $n = 4$ by Chang, Gursky and Yang \cite{CGY02-AnnM}, for locally conformally flat manifolds and all $k$'s Li and Li \cite{LiLi03} (the existence part was also proved  independently by Guan and Wang \cite{GW03-JRAM}), for $k > n/2$ by Gursky and Viaclovsky \cite{GV07} and for $k = n/2$ by Li and Nguyen \cite{LiNgPoorMan}. For related works, see also \cite{GeWang06, GurskyViaclovsky04-AiM, STW07, TW09} and the references therein.

Our proof of Theorem \ref{main'} follows closely the strategy used in Li \cite{Li95-JDE}, Chang and Yang \cite{ChangYang91-Duke}, and Chang, Han and Yang \cite{CHY11-CVPDE} with some twists. To begin with, by a result of Li and Nguyen \cite{LiNgPoorMan} for $k \geq n/2$, in a hypothetical situation where \eqref{uniform estimate} is violated, one has at most a single isolated blow-up point. More precisely, if $\{v_i\}$ is a sequence of solutions to \eqref{yq} with $\max v_i \rightarrow \infty$, then there exist points $p_i$ with $v_i(p_i) = \max v_i$, some $\delta_i \rightarrow 0^+$ and some constant $C > 0$ independent of $i$ such that
\begin{align*}
&v_i(x) 
	\leq C\textrm{dist}_{g_0}(x,p_i)^{-\frac{n-2}{2}} \text{ for all } x \neq p_i,\\
&\sup_{\textrm{dist}_{\roundg}(x,p_i) > r} v_i 
	\rightarrow 0 \text{ for all small } r > 0,\\
&\textrm{Volume}(g_{v_i}) 
	\rightarrow \textrm{Volume}(\roundg),\\
&v_i(x)
	\leq C\textrm{dist}_{g_0}(x,p_i)^{-(n-2 - \delta_i)}  \min v_i
		\text{ when } \textrm{dist}_{g_0}(x,p_i) \geq \frac{1}{C} v_i(p_i)^{-\frac{2}{n-2}}.
\end{align*}
One new ingredient for our argument is the following sharp fall-off behavior of solutions as one moves away from the blow-up points $p_i$ (see Lemma \ref{Lem:20VI20-KeyNew}):
\begin{equation}
v_i(x) \leq \frac{C}{v_i(p_i)} \textrm{dist}_{g_0}(x,p_i)^{-(n-2)} \text{ when } \textrm{dist}_{g_0}(x,p_i) \geq \frac{1}{C} v_i(p_i)^{-\frac{2}{n-2}}.
	\label{Eq:20VI20-E1}
\end{equation}

When $k > n/2$, we establish estimate \eqref{Eq:20VI20-E1} using the Liouville-type theorem and H\"older estimates for the differential inclusion $\lambda(A_{g_v}) \in \bar\Gamma_k$. The case $k = n/2$ is more delicate and requires some different ideas. In fact, we provide two different proofs for \eqref{Eq:20VI20-E1} in this case (see Subsection \ref{SSec:L22k=n/2}).

For future use, we state the main additional ingredients for the proof of \eqref{Eq:20VI20-E1} using notations in the Euclidean space. Let $\mathring{g}$ denote the Euclidean metric on $\RR^n$ and $A^u$ denote the $(1,1)$-Schouten tensor of $u^{\frac{4}{n-2}}\mathring{g}$, 
\begin{equation}
A^u = -\frac{2}{n-2}u^{-\frac{n+2}{n-2}}\nabla^{2}u+\frac{2n}{(n-2)^{2}}u^{-\frac{2n}{n-2}}du\otimes  d u - \frac{2}{(n-2)^{2}}u^{-\frac{2n}{n-2}}|du|^{2} \mathring{g},
	\label{Eq:AuDef}
\end{equation}
where $\nabla$ is the covariant derivative of $\mathring{g}$.

For our first proof of \eqref{Eq:20VI20-E1}, we make use of an estimate usually referred to as `small energy implies boundedness'. This estimate extends results of Han \cite{Han04} in the case $k= 2$, $n = 4$ and Gonzalez \cite{Gonzalez06-RemSing} in the case $k < n/2$.

\begin{theorem}\label{Thm:20VI20-deltaEnergy}
Let $n \geq 3$ and $1 \leq k \leq n/2$ and $\Omega$ be a non-empty open subset of $\RR^n$. There exists a constant  $\delta > 0$ depending only on $n$ such that if $u \in C^2(\Omega)$ satisfies
\begin{align}
&\sigma_k(\lambda(A^u)) \leq 1, \quad \lambda(A^u) \in \bar\Gamma_k, \quad u > 0 \text{ in } \Omega,
	\label{Eq:20VI20-X1}
\end{align}
and $\int_{\Omega} u^{\frac{2n}{n-2}}\,dx < \delta$, then
\[
\textrm{dist}(x,\partial\Omega)^{\frac{n-2}{2}} u(x) \leq 4^{\frac{n-2}{2}} \text{ in } \Omega.
\]
\end{theorem}

\begin{remark}
In \cite{Gonzalez06-RemSing}, when $k < n/2$, a stronger version of the above theorem on punctured balls is proved. It was pointed out in \cite{Han04} that this no longer holds when $k = n/2$. More precisely, in view of results in \cite{ChangHanYang05-JDE}, for every $\delta > 0$, there exists a (rotationally symmetric) function $u \in C^2(B_2 \setminus \{0\})$ such that
\[
\sigma_{n/2}(\lambda(A^u)) = 1, \quad \lambda(A^u) \in \Gamma_{n/2}, \quad u > 0 \text{ in } B_2 \setminus \{0\},
\]
and $\int_{B_2} u^{\frac{2n}{n-2}}\,dx < \delta$ but $\sup_{B_1 \setminus \{0\}} u = +\infty$.
\end{remark}

For our second proof of \eqref{Eq:20VI20-E1}, we make use of the following local boundedness estimate for viscosity subsolutions to the $\sigma_k$-Yamabe equation, which is of independent interest. We refer the readers to \cite{Li09-CPAM} for the definition of viscosity (sub-/super-)solutions.

\begin{lemma}\label{Thm:20VI20-BetterD=>FX}
Let $n \geq 4$ be even and $B_1$ be the unit ball in $\RR^n$. Suppose that $u \in C_{loc}(\bar B_1 \setminus \{0\})$ is positive and satisfies
\begin{align}
&\lambda(A^u) \in \RR^n \setminus \{\lambda \in \Gamma_{n/2}:  \sigma_{n/2}(\lambda) > 1\} \text{ in } B_1 \setminus \{0\} \text{ in the viscosity sense},
	\label{Eq:04III20-A1X}\\
& \limsup_{|x|\rightarrow 0} u(x) < \infty.
	\label{Eq:04III20-A2X}
\end{align}
For given $C_0 > 0$ and $\alpha \in (0,1]$, there exists $C_1 = C_1(n,C_0,\alpha)$ such that if $
u(x) \leq C_0|x|^{-\frac{n-2}{2}(1 - \alpha)}$ in $\bar B_1 \setminus \{0\}$, then
\[
u(x) \leq C_1 \text{ in } \bar B_1 \setminus \{0\}.
\]
\end{lemma}

We note that in the above lemma, condition \eqref{Eq:04III20-A2X} cannot be dropped. This is because, for any $\beta \in (0,n-2]$, the function $u(x) = |x|^{-\beta}$ satisfies \eqref{Eq:04III20-A1X} on $B_1 \setminus \{0\}$ but is not locally bounded at the origin.

We point out that if $\sigma_{k}(\lambda(A^u)) < 1$ in $B_1 \setminus \{0\}$ for some $k < n/2$, then \eqref{Eq:04III20-A1X} holds.

We refer the readers to \cite{Gonzalez05, Gonzalez06-RemSing, HanLiLi21-CPAM, HLT10} and the references therein for further studies on the asymptotic analysis for solutions of the $\sigma_k$-Yamabe equation with singularities.

Last but not least, we note that it is interesting to see if a purely local version of \eqref{Eq:20VI20-E1} holds in a local blow-up situation when $k = n/2$. More precisely, if $u_i$ satisfies $\sigma_{n/2}(\lambda(A^{u_i})) = K_i(x) > 0$, $\lambda(A^{u_i}) \in \Gamma_{n/2}$ in $B_1 \subset \RR^n$, $K_i \rightarrow K > 0$ in $C^2(\bar B_1)$ and $u_i(0) = \max_{B_1} u_i \rightarrow \infty$, does it hold e.g. on $\partial B_{1/2}$ that
\[
u_i(x) \leq \frac{C}{u_i(0)} \text{ for some constant $C$ independent of $i$}?
\]
Our proof of \eqref{Eq:20VI20-E1} for $k = n/2$ uses the fact that $v_i$ satisfies \eqref{yq} which is valid on the whole of $\SSphere^n$. We note that, when $k > n/2$, the answer to the above question is affirmative, as proved in Subsection \ref{SSec:L22k>n/2}.

The rest of our paper is organized as follows. We start in Section \ref{Sec:Cptness} with the proof of the main estimate \eqref{uniform estimate} in Theorem \ref{main'} by establishing the optimal decay estimate \eqref{Eq:20VI20-E1}, assuming Lemma \ref{Thm:20VI20-BetterD=>FX} or Theorem \ref{Thm:20VI20-deltaEnergy}. In Section \ref{Sec:LBdedness}, we give the proof of Lemma \ref{Thm:20VI20-BetterD=>FX}. Section \ref{Sec:epsEnergy}  is devoted to the proof of Theorem \ref{Thm:20VI20-deltaEnergy}. In Section \ref{Sec:ProofMain}, we complete the proof of Theorem \ref{main'}. We include also an appendix where certain convexity property of the Schouten tensor is proved and used to compliment certain arguments made in the body of the paper.

\section{A compactness result}\label{Sec:Cptness}

In this section, we prove the compactness estimate \eqref{uniform estimate}, with the help of Theorem \ref{Thm:20VI20-deltaEnergy} or Lemma \ref{Thm:20VI20-BetterD=>FX}. 

\begin{theorem}\label{thm9999}
Under the assumptions of Theorem \ref{main'}, estimate \eqref{uniform estimate} holds with a constant $C_*$ depending only on $n$, an upper bound of $|\ln K|$, a positive lower bound of $\frac{1}{\|\nabla_{\roundg}^2 K\|_{C^0(\SSphere^n)}} [|\nabla_{g_0} K|_{g_0} + |\Delta_{g_0} K|]$, and the modulus of continuity of $\frac{1}{\|\nabla_{\roundg}^2 K\|_{C^0(\SSphere^n)}} \nabla_{\roundg}^2 K$.
\end{theorem}

We note that in the above theorem $K$ is non-constant by assumptions, and hence $\|\nabla_{\roundg}^2 K\|_{C^0(\SSphere^n)} > 0$.

\begin{proof} In view of first and second derivative estimates for the $\sigma_k$-Yamabe equation (see \cite{Chen05, GW03-IMRN}, \cite[Theorem 1.10]{Li09-CPAM}, \cite[Theorem 1.20]{LiLi03}, \cite{Wang06}), it suffices to show that 
\begin{equation}
v \leq C_* \text{ for all $C^2$ positive solutions $v$ of \eqref{yq}}
	\label{Eq:26VII20-E1}
\end{equation}
where $C_*$ is as indicated in the statement of the theorem.

We use a combination of ideas from \cite{CHY11-CVPDE, Li95-JDE, LiNgPoorMan} with one new input (Lemma \ref{Lem:20VI20-KeyNew} below). Suppose by contradiction that \eqref{Eq:26VII20-E1} does not hold. Then
\begin{enumerate}[(i)]
\item there exist positive $C^2$ functions $K_i$ such that $|\ln K_i|$ is uniformly bounded, $\frac{1}{\|\nabla_{\roundg}^2 K_i\|_{C^0(\SSphere^n)}} [|\nabla_{g_0} K_i|_{g_0} + |\Delta_{g_0} K_i|]$ is uniformly bounded from below by a positive constant, and $\frac{1}{\|\nabla_{\roundg}^2 K_i\|_{C^0(\SSphere^n)}} \nabla_{\roundg}^2 K_i$ is uniformly equi-continuous, 
\item there exist positive $C^2$ functions $v_{i}$ satisfying (\ref{yq}) with $K$ replaced by $K_i$, and a sequence of points $\{P_{i}\}_{i=1}^{\infty}\subset\SSphere^{n}$ such that $v_{i}(P_{i})=\max v_{i}\rightarrow\infty$ as $i\rightarrow\infty$.
\end{enumerate}

Composing both $v_i$ and $K_i$ with a rotation of the sphere if necessary, we may assume without loss of generality that all $P_i$ are the same: $P_i = P$, which we will conveniently taken at the south pole. This point will be called the blow-up point of $\{v_i\}$.

We claim that $\|\nabla_{\roundg}^2 K_i\|_{C^0(\SSphere^n)}$ is uniformly bounded. Indeed, if $\|\nabla_{\roundg}^2 K_{i_j}\|_{C^0(\SSphere^n)}$ were to diverge to $\infty$ along some subsequence, then, the sequence $\frac{1}{\|\nabla_{\roundg}^2 K_{i_j}\|_{C^0(\SSphere^n)}} K_{i_j}$ would converge uniformly to $0$. As this sequence is pre-compact in $C^2$ by Ascoli-Arzela's theorem, the convergence would hold in fact in $C^2$. This would lead to a contradiction to the statement that $\frac{1}{\|\nabla_{\roundg}^2 K_i\|_{C^0(\SSphere^n)}} [|\nabla_{g_0} K_i|_{g_0} + |\Delta_{g_0} K_i|]$ is uniformly bounded from below by a positive constant. The claim is proved.

In view of the above claim and (i), we may assume without loss of generality that $K_i$ converges in $C^2$ to some positive $C^2$ function $K_\infty$.

By Step 1 in the proof of \cite[Theorem 1.3]{LiNgPoorMan}, the sequence $\{v_{i}\}_{i=1}^{\infty}$ has precisely only one isolated blow-up point. More precisely, there exists $C > 0$ such that
\begin{equation}
v_i(x) \leq C d(x,P)^{-\frac{n-2}{2}} \text{ for all } x \in \SSphere^n \setminus \{P\},
	\label{Eq:05III20-Y1}
\end{equation}
where $d$ denotes the distance function induced by $g_{0}$. 

Let $\Phi: \RR^n \rightarrow \SSphere^n $ be the inverse of the stereographic projection with $P$ being the south pole to the equatorial plane of $\SSphere^{n}$, that is,
\[
x_{i} =\frac{2y_{i}}{1+|y|^{2}} \text{ for }i=1,\ldots,n, 
\text{ and }
x_{n+1} =\frac{|y|^{2}-1}{|y|^{2}+1}.
\]
Let
\begin{equation*}
u_{i}(y)=\left(\frac{2}{1+|y|^{2}}\right)^{\frac{n-2}{2}}v_{i}(x).
\end{equation*}
Then $0$ is a maximum point of of $u_{i}$, $u_{i}(0)\rightarrow\infty$ as $i\rightarrow\infty$ and
\begin{equation}\label{genqian}
\sigma_{k}(\lambda(A^{u_{i}}))=K_{i}(\Phi(y)),~~\lambda(A^{u_{i}})\in\Gamma_{k},~~~~\mbox{in }\RR^{n},
\end{equation}
where $A^{u_i}$ is as in \eqref{Eq:AuDef}.

Let $a = 2^{-\frac{n-2}{4}} \Big(\begin{array}{c}n\\k\end{array}\Big)^{-\frac{n-2}{4k}} K_{\infty}(P)^{\frac{n-2}{4k}} > 0$. By working with the sequence $\{c_{n,k}a v_i\}$ instead of $\{v_i\}$ for some positive constant $c_{n,k}$ depending only on $n$ and $k$, we may assume without loss of generality that $a = 1$. 

Let $\lambda_i =  u_i(0)^{\frac{2}{n-2}}$ and define
\[
\tilde u_i(z) = \frac{1}{u_i(0)} u_i\Big( \lambda_i^{-1} z\Big).
\]
Then
\[
\sigma_{k}(\lambda(A^{\tilde u_{i}}))= K_{i}(\Phi(  \lambda_i^{-1} z)), \quad \lambda(A^{\tilde u_{i}})\in\Gamma_{k} \quad \mbox{in }\RR^{n}.
\]

Passing to a subsequence if necessary and using the Liouville-type theorem \cite[Theorem 1.3]{LiLi05}, for every $\vareps_i \rightarrow 0^+$ we can select $R_{i}\rightarrow\infty$ with $R_i < |\ln \varepsilon_i|$ and $R_{i}/\lambda_i \rightarrow 0$ as $i \rightarrow \infty$ such that
\begin{equation}
\|\tilde u_{i}(z)- (1+|z|^{2})^{\frac{2-n}{2}}\|_{C^{2}(B_{2R_{i}})}\leq\varepsilon_{i}.
	\label{as}
\end{equation}

\begin{lemma}[Optimal decay estimate] \label{Lem:20VI20-KeyNew}
With the notations as above, there exists a positive constant $C$ independent of $i$ such that
\begin{equation}\label{zhongyao}
u_{i}(y)\leq Cu_{i}^{-1}(0)|y|^{2-n} \text{ for all } y \in \RR^n.
\end{equation}
\end{lemma}

Let us assume this lemma for the moment and go on to complete the proof using the Kazdan-Warner and Pohozaev identities as in \cite{ChangYang91-Duke, Li95-JDE}.

Let $\tilde K_i = \frac{1}{\|\nabla_{\roundg}^2 K_i\|_{C^0(\SSphere^n)}} (K_i \circ \Phi - K_i(P))$. By (\ref{genqian}) and \cite[Theorem 1]{Han06} (see also \cite{Viac00-AMS}), we have that
\begin{equation}
0= \frac{1}{\|\nabla_{\roundg}^2 K_i\|_{C^0(\SSphere^n)}}  \int_{\RR^{n}} \partial_\ell \sigma_{k}(\lambda(A^{u_{i}})) u_{i}^{\frac{2n}{n-2}}\,dy =\int_{\RR^{n}} \partial_\ell\tilde K_i u_{i}^{\frac{2n}{n-2}}\,dy, \quad 1 \leq \ell \leq n,
	\label{rushan}
\end{equation}
and
\begin{equation}
0= \frac{1}{\|\nabla_{\roundg}^2 K_i\|_{C^0(\SSphere^n)}}  \sum\limits_{\ell=1}^{n}\int_{\RR^{n}}y_{\ell}\partial_\ell \sigma_{k}(\lambda(A^{u_{i}})) u_{i}^{\frac{2n}{n-2}}\,dy = \int_{\RR^{n}}y \cdot \nabla\tilde K_i u_{i}^{\frac{2n}{n-2}}\,dy.
	\label{rushan3}
\end{equation}

Fix some $r_0 > 0$ for the moment and let us consider \eqref{rushan}. Note that (i) implies that $\tilde K_i = \frac{1}{\|\nabla_{\roundg}^2 K_i\|_{C^0(\SSphere^n)}} (K_i  - K_i(P))$ is pre-compact in $C^2$. It follows that $|\nabla \tilde K_i(y)| = \frac{O(1)}{1 + |y|^2}$ on $\RR^n$, where here and below the implicit constants in the big `$O$' and little `$o$' notations are independent of $i$. Thus, by \eqref{rushan} and \eqref{zhongyao}
\begin{equation}
0 =\int_{|y | \leq r_0} \partial_\ell\tilde K_i u_{i}^{\frac{2n}{n-2}}\,dy + O(\lambda_i^{-n}), \quad \ell=1,\ldots,n.
	\label{Eq:20VI20-M1}
\end{equation}
Observe that, by \eqref{as} and \eqref{zhongyao}, if $q:\RR^n \setminus \{0\} \rightarrow \RR$ is a homogeneous function of degree $d \in [0,n)$, then
\begin{equation}
\lim_{i \rightarrow \infty}  \lambda_i^d \int_{|y | \leq r_0} q(y) u_i^{\frac{2n}{n-2}}(y)\,dy  = \int_{\RR^n}  \frac{q(z)}{(1+ |z|^2)^{n}}\,dz.
	\label{Eq:20VI20-HomConv}
\end{equation}
Using Taylor's theorem, we write
\[
\partial_\ell \tilde K_i(y) =\partial_\ell \tilde K_i(0) + \sum_{p=1}^n \partial_p \partial_{\ell}\tilde K_i(0) y_p + o_{r_0}(1) |y| \text{ for } |y| \leq r_0
\]
where $o_{r_0}(1) \rightarrow 0$ as $r_0 \rightarrow 0$. Plugging this into \eqref{Eq:20VI20-M1} and using \eqref{Eq:20VI20-HomConv} we are led to
\begin{equation}
0 = M^{(i)}\partial_\ell \tilde K_i(0)
	+ M^{(i)}  \sum_{p=1}^n \partial_p \partial_{\ell} \tilde K_i(0) \mu^{(i)}_p + o_{r_0}(1) \lambda_i^{-1}, \quad \ell=1,\ldots,n,
		\label{Eq:20VI20-KWP1}
\end{equation}
where $M^{(i)}$ and $\mu^{(i)}_p$'s are given by
\begin{align}
M^{(i)} 
	&= \int_{|y| \leq r_0} u_{i}^{\frac{2n}{n-2}}\,dy \stackrel{\eqref{Eq:20VI20-HomConv}}{\geq} \frac{1}{C},
	\label{Eq:22VI20-MassB}\\
\mu^{(i)}_p
	&= \frac{1}{M^{(i)}}  \int_{|y| \leq r_0} y_p u_{i}^{\frac{2n}{n-2}}\,dy \stackrel{\eqref{Eq:20VI20-HomConv}}{=} \frac{o(1)}{\lambda_i}.
	\label{Eq:20VI20-1stMoment}
\end{align}
We thus have
\begin{equation}
|\nabla K_i(0)| = \frac{o_{r_0}(1)}{\lambda_i} \text{ as } i \rightarrow \infty.
	\label{Eq:20VI20-GradDecay}
\end{equation}

We now turn to \eqref{rushan3}. By the same argument, we have
\begin{equation}
0 = M^{(i)} \sum_{\ell = 1}^n \partial_\ell \tilde K_i(0) \mu^{(i)}_\ell
	+ M^{(i)}  \sum_{\ell, p=1}^n \partial_p \partial_{\ell} \tilde K_i(0) \mu^{(i)}_{\ell p} + o_{r_0}(1) \lambda_i^{-2}.
		\label{Eq:20VI20-KWP2}
\end{equation}
where $\mu^{(i)}_{\ell p}$'s are given by
\begin{align}
\mu^{(i)}_{\ell p}
	&= \frac{1}{M^{(i)}}  \int_{|y| \leq r_0} y_\ell y_p u_{i}^{\frac{2n}{n-2}}\,dy \stackrel{\eqref{Eq:20VI20-HomConv}, \eqref{Eq:22VI20-MassB}}{=} \frac{O(1) \delta_{\ell p} + o(1)}{\lambda_i^2} .
	\label{Eq:20VI20-2ndMoment}
\end{align}
Combining \eqref{Eq:20VI20-KWP1} and  \eqref{Eq:20VI20-KWP2} we obtain
\[
\sum_{\ell, p=1}^n \partial_p \partial_{\ell} \tilde K_i(0) (\mu^{(i)}_{\ell p} - \mu^{(i)}_{\ell }\mu^{(i)}_{p}) =  o_{r_0}(1)\lambda_i^{-2}.
\]
Recalling \eqref{Eq:20VI20-1stMoment} and \eqref{Eq:20VI20-2ndMoment}, we obtain
\begin{equation}
\Delta_{\mathring{g}} \tilde K_i(0) = o_{r_0}(1) \text{ as } i \rightarrow \infty.
	\label{Eq:20VI20-LapDecay}
\end{equation}
This together with \eqref{Eq:20VI20-GradDecay} implies that $\frac{1}{\|\nabla_{\roundg}^2 K_i\|_{C^0(\SSphere^n)}} [|\nabla_{g_0} K_i(0)|_{g_0} + |\Delta_{g_0} K_i(0)|]|$ converges to $0$ as $i \rightarrow \infty$, which gives a contradiction to (i).

We have proved Theorem \ref{thm9999} assuming the correctness of the optimal decay estimate in Lemma \ref{Lem:20VI20-KeyNew}. To finish, we need to give the proof of Lemma \ref{Lem:20VI20-KeyNew}. This can be found in Subsections \ref{SSec:L22k=n/2} and \ref{SSec:L22k>n/2} below.
\end{proof}

\subsection{Proof of the optimal decay estimate when $k = n/2$}\label{SSec:L22k=n/2}

As mentioned in the introduction, we will give two proofs of Lemma \ref{Lem:20VI20-KeyNew}. Let us start with the first proof.

\begin{proof}[First proof of Lemma \ref{Lem:20VI20-KeyNew} when $k = n/2$] 
Let $\tilde u_{i}^{*}(z)=|z|^{2-n}\tilde u_{i}(\frac{z}{|z|^{2}})$. Then $\tilde u_i^*$ satisfies 
\[
\sigma_k(\lambda(A^{\tilde u_i^*})) \leq   \sup_i \max K_{i} =: b , \quad \lambda(A^{\tilde u_i^*}) \in \Gamma_k \text{ on } \RR^n,
\]
and, by (\ref{as}), we have that for any $r >0$,
\begin{equation*}
\int_{\RR^{n}\setminus B_r}(\tilde u_{i}^{*})^{\frac{2n}{n-2}}(z)\,dz =\int_{B_{\frac{1}{r}}}\tilde u_{i}^{\frac{2n}{n-2}}(z)\,dz\rightarrow\int_{B_{\frac{1}{r}}}(1+|z|^{2})^{-n}\,dz \quad\mbox{as }i\rightarrow\infty.
\end{equation*}
By \cite[Remark 3.2]{LiNgPoorMan}, we have that
\begin{equation*}
\int_{\RR^{n}}(\tilde u_{i}^{*})^{\frac{2n}{n-2}}(z)\,dz =\int_{\RR^n}\tilde u_{i}^{\frac{2n}{n-2}}(z)\,dz \rightarrow\int_{\RR^{n}}(1+|z|^{2})^{-n}\,dz \quad \mbox{as }i\rightarrow\infty.
\end{equation*}
Thus,
\begin{equation*}
\int_{B_{r}}(\tilde u_{i}^{*})^{\frac{2n}{n-2}}(z)\,dy \rightarrow\int_{\RR^{n}\setminus B_{\frac{1}{r}}}(1+|z|^{2})^{-n}dz \quad \mbox{as }i\rightarrow\infty.
\end{equation*}
In particular, we can pick a sufficiently small $r$ so that the integral on the left hand side is smaller than the constant $\delta$ in Theorem \ref{Thm:20VI20-deltaEnergy}. We then obtain from Theorem \ref{Thm:20VI20-deltaEnergy} that
\begin{equation*}
\tilde u_{i}^{*}\leq C_n\,r^{-\frac{n-2}{2}} \text{ in } B_{r/2}.
\end{equation*}
Returning to $u_i$, we obtain
\[
u_i(y) \leq \frac{C_n r^{-\frac{n-2}{2}}}{u_i(0)  |y|^{n-2}} \text{ for } |y | > 2 r^{-1}  /\lambda_i.
\]
Estimate (\ref{zhongyao}) then follows from the above estimate and \eqref{as}.
\end{proof}

We now turn to the 
\begin{proof}[Second proof of Lemma \ref{Lem:20VI20-KeyNew} when $k = n/2$]
In view of \eqref{as}, it suffices to show that
\begin{equation}
\tilde u_i(z) \leq \frac{C}{|z|^{n-2}} \text{ for } |z| \geq 1.
	\label{Eq:20VI20-H1}
\end{equation}
By \eqref{Eq:05III20-Y1}, we have $\tilde u_i(z) \leq C|z|^{-\frac{n-2}{2}}$ on $\RR^n$. By known Harnack estimates (see e.g. \cite[Theorem 1.10]{Li09-CPAM}), this implies that
\begin{equation}
\sup_{\partial B_r} \tilde u_i \leq C_1 \inf_{\partial B_r} \tilde u_i \text{ for every } r \geq 1.
	\label{Eq:20VI20-H2}
\end{equation}
The desired estimate \eqref{Eq:20VI20-H1} then follows from \eqref{as}, \eqref{Eq:20VI20-H2} and Lemma \ref{Prop:04III20-P3} below.
\end{proof}

\begin{lemma}\label{Prop:04III20-P3}
Suppose that $0<u \in C^0_{loc}(\{|x| \geq 1\})$ satisfies \eqref{Eq:20VI20-H2} and
\begin{align}
&\sigma_{n/2}(\lambda(A^{u})) \leq 1, \lambda(A^{u}) \in \bar\Gamma_{n/2} 
 \text{ in } \RR^n \setminus \bar B_1,
	\label{Eq:03VIII20-R1}\\
& \limsup_{|x|\rightarrow \infty} {u}(x) |x|^{n-2} < \infty,
\nonumber
\end{align}
where \eqref{Eq:03VIII20-R1} is understood in the viscosity sense.

For any $\varkappa > 0$ and $R > R' \geq 1$, there exists $\delta = \delta(n,\varkappa, R, R') > 0$ and $C_2 =  C_2(n,\varkappa, R, R', C_1) > 0$ such that if
\[
\sup_{\{R' \leq |x| \leq R\}} \Big|u(x) - \varkappa (1 + |x|^2)^{-\frac{n-2}{2}}\Big| \leq \delta,
\]
then
\[
u(x) \leq C_2|x|^{-(n-2)} \text{ in } \{|x| \geq 1\}.
\]
\end{lemma}

\begin{proof} For $r = |x| \geq 1$, let $U(r) = \varkappa (1 + r^2)^{-\frac{n-2}{2}}$ and $\bar u(r) = \inf_{\partial B_r} u$. Clearly $\bar u$ is continuous. Note also that if we define, for $O \in O(n)$, the set of all $n\times n$ orthogonal matrices, $u_O(x) = u(Ox)$, then $\lambda(A^{u_O})  \in \bar\Gamma_{n/2}$ in $\{|x| > 1\}$ in the viscosity sense and $\bar u = \inf \{u_O: O \in O(n)\}$. Hence $\lambda(A^{\bar u}) \in \bar\Gamma_{n/2}$ in $\{|x| > 1\}$ in the viscosity sense. By \cite[Lemma 2.7]{LiNgBocher}, $\bar u \in C^{0,1}_{loc}(\{|x| \geq 1\})$ and we have for $c > d > 1$,
\begin{equation}
0 \leq \ln \frac{\bar u(d)}{\bar u(c)} \leq (n-2) \ln \frac{c}{d}.
	\label{Eq:04III20-M1}
\end{equation}

Fix some $R' < s_1 < s_2 < R$. Note that
\[
\frac{U(s_1)}{U(s_2)} = \Big(\frac{1 + s_2^2}{1 + s_1^2}\Big)^{\frac{n-2}{2}} > \Big(\frac{s_2}{s_1}\Big)^{\frac{n-2}{2}}.
\]
Let
\[
\beta = \frac{1}{2}\Big\{\frac{2}{n-2}\Big(\ln \frac{s_2}{s_1}\Big)^{-1} \ln \frac{U(s_1)}{U(s_2)} - 1\Big\} > 0.
\]
Then, there exists small $\delta > 0$ such that when $\|u - U\|_{L^\infty(\{R' \leq |x| \leq R\})} < \delta$, we have
\[
\alpha := \frac{2}{n-2}\Big(\ln \frac{s_2}{s_1}\Big)^{-1} \ln \frac{\bar u (s_1)}{\bar u (s_2)} - 1 > \beta > 0.
\]
Note that by \eqref{Eq:04III20-M1}, $\alpha \leq 1$. Fix such $\delta$ from now on.

Using $\|u - U\|_{L^\infty(\{R' \leq |x| \leq R\})} < \delta$, \eqref{Eq:04III20-M1} and  \eqref{Eq:20VI20-H2} , we have that
\begin{equation}
u \leq C \bar u \leq C \text{ in } \{1 \leq |x| \leq R\}.
	\label{Eq:05III20-X2}
\end{equation}

Now consider the function
\[
V(r) = \bar u(s_1) \Big(\frac{r}{s_1}\Big)^{-\frac{n-2}{2}(1 + \alpha)}
\]
which satisfies $\lambda(A^V) \in \partial\Gamma_{n/2}$ on $\RR^n \setminus \{0\}$ (see \cite[Theorem 1.6]{LiNgBocher}), $V(s_1) = \bar u(s_1)$ and $V(s_2) = \bar u(s_2)$.

We claim that $V(r) \geq \bar u(r)$ for all $r > s_2$. Indeed, suppose by contradiction this did not hold. Then we would have for some $s_3 > s_2$ that $V(s_3) < \bar u(s_3)$. Applying the comparison principle (see \cite[Lemma 2.8]{LiNgBocher} and \cite[Theorem 1.7 (b)]{LiNgWang18-CVPDE}) to $V$ and $\bar u$ in the annulus $\{s_1 \leq |x| \leq s_3\}$, we would arrive at $V < \bar u$ in $\{s_1 < |x| < s_3\}$ which is impossible as $V(s_2) = \bar u(s_2)$. The claim is proved.

In view of  \eqref{Eq:20VI20-H2}  and \eqref{Eq:05III20-X2}, it follows that
\[
u(x) \leq C|x|^{-\frac{n-2}{2}(1 + \alpha)} \text{ in } \{|x| \geq 1\}.
\]
The conclusion follows from the local boundedness estimate Lemma \ref{Thm:20VI20-BetterD=>FX} applied to the Kelvin transformation $\hat u(x) = \frac{1}{|x|^{n-2}}u(\frac{x}{|x|^2})$ of $u$.
\end{proof}

\subsection{Proof of the optimal decay estimate when $k > n/2$}\label{SSec:L22k>n/2}

When $k > n/2$, our proof of Lemma \ref{Lem:20VI20-KeyNew} is simpler, and uses the H\"older continuity for functions satisfying $\lambda(A^u) \in \bar \Gamma_k$. We also note that there are other methods to prove Lemma \ref{Lem:20VI20-KeyNew} in this case; see Appendix \ref{App:Misc} for one such argument.

\begin{proof}[Proof of Lemma \ref{Lem:20VI20-KeyNew} when $k > n/2$] 
By \eqref{as} and the Harnack estimate (cf. \eqref{Eq:20VI20-H2}), it suffices to show, for some $C > 0$ independent of $i$, that
\begin{equation}
\min_{\partial B_r} u_{i}\leq Cu_{i}^{-1}(0) r^{2-n} ~~\forall~r\geq R_i/\lambda_i,
	\label{Eq:23VII20-X1}
\end{equation}
where $\lambda_i = u_i(0)^{\frac{2}{n-2}}$ and $R_i$ is as in \eqref{as}.

For any $r \geq R_i/\lambda_i$, define $\bar u_{i}(\xi):=r^{\frac{n-2}{2}}u_{i}(2r\xi )$, $\forall \xi\in B_{1}$, and fix some $\xi_i$ such that $|\xi_i| = \frac{1}{2\lambda_i r}$. Then $\lambda(A^{\bar u_{i}})\in\Gamma_{k}$ in $B_{1}$ and
\begin{equation}
|\xi_i|^{\frac{n-2}{2}} \bar u_i(0) = 2^{-\frac{n-2}{2}}.
	\label{Eq:23VII20-M0}
\end{equation}
Moreover, in view of \eqref{as}, we have $|\xi_i|^{\frac{n-2}{2}}\bar u_i(\xi_i) = 2^{-\frac{n-2}{2}} \tilde u_i( \frac{\xi_i} {|\xi_i|}) = 2^{-(n-2)} + o(1)$ and so, for some $i_0$ independent of $r$,
\begin{equation}
|\xi_i|^{\frac{n-2}{2}}\bar u_i(\xi_i)  \leq 2^{-\frac{3(n-2)}{4}} \text{ for all } i \geq i_0.
	\label{Eq:23VII20-M1}
\end{equation}
As $|\xi_i| \leq \frac{1}{2R_i}  \leq \frac{1}{2}$ for large $i$, it thus follows, in view of \cite[Theorem 1.4, equation (1.14)]{LiNgBocher}, that the function $w_i = \bar u_i^{\frac{n-2k}{(n-2)k}}$ satisfies
\[
\sup_{\partial B_{1/2}} w_i 
	\geq \frac{|w_i(0) - w_i(\xi_i)|}{C|\xi_i|^{\frac{2k-n}{k}}} 
	\stackrel{\eqref{Eq:23VII20-M0},\eqref{Eq:23VII20-M1}}{\geq} \frac{1}{C |\xi_i|^{\frac{2k-n}{2k}}}
	\stackrel{\eqref{Eq:23VII20-M0}}{\geq} \frac{1}{C w_i(0)} \text{ for large } i,
\]
where the constant $C$ depends only on $n$ and $k$. Returning to $u_i$ we obtain \eqref{Eq:23VII20-X1}.
 \end{proof}

\section{Proof of Lemma \ref{Thm:20VI20-BetterD=>FX}}\label{Sec:LBdedness}

\begin{proof} For $R > 1$, let $\tilde u = \tilde u_R(x) = c_n \frac{R^{n-2}}{|x|^{n-2}} u\big(\frac{R^2x}{|x|^2}\big)$ for $|x| \geq 1$ where $c_n$ is a normalization constant depending only on $n$, suitable chosen so that $\tilde u$ satisfies 
\begin{align}
&\lambda(A^{\tilde u}) \notin \bar\Gamma_{n/2} \text{ or } \sigma_{n/2}(\lambda(A^{\tilde u})) \leq 2^{-n/2} \Big(\begin{array}{c}n\\n/2\end{array}\Big) 
 \text{ in } \RR^n \setminus \bar B_1,
	\label{Eq:04III20-A1}\\
& \limsup_{|x|\rightarrow \infty} {\tilde u}(x) |x|^{n-2} < \infty,
	\label{Eq:04III20-A2}\\
& \tilde u(x) \leq C_0 R^{-\frac{(n-2)\alpha}{2}} |x|^{-\frac{n-2}{2}(1 + \alpha)}.
	\label{Eq:04III20-A3td}
\end{align}
Here \eqref{Eq:04III20-A1} is understood in the viscosity sense.

By \cite[Theorem 1, Case I.1 and I.3]{ChangHanYang05-JDE}, for $a \geq 0$, there exists a unique solution to the problem
\begin{align*}
\sigma_{n/2}(\lambda(A^{V_{a}})) 
	&= 2^{-n/2} \Big(\begin{array}{c}n\\n/2\end{array}\Big),\qquad \lambda(A^{V_{a}}) \in \Gamma_{n/2} \text{ in } \RR^n \setminus \{0\},\\
V_{a}(1)
	&= e^{-\frac{n-2}{2}a}
		 \text{ and } \frac{d}{dr}\Big|_{r=1} \ln\Big(r^{\frac{n-2}{2}}V_{a}\Big) = 0.
\end{align*}
Furthermore, with $h(a) = 1 - e^{-na} \in [0,1)$ and $\gamma(a) = \frac{n-2}{2}\big[1 + \big(1 - h(a)^{2/n}\big)^{1/2}\big] \in (\frac{n-2}{2}, n-2]$, we have
\begin{equation}
0 < \lim_{r \rightarrow \infty} V_{a} r^{\gamma(a)}  < \infty.
	\label{Eq:05III19-T1}
\end{equation}
Moreover, the limit is uniform for $a$ belonging to a compact interval. Note that $\gamma'(a) < 0$ is monotone decreasing and $\gamma(0) = n - 2$ and $\gamma(\infty) = \frac{n-2}{2}$.

\medskip
\noindent{\it Step 1:} We claim that there exist large $a_0 > 0$ and large $R > 0$ such that 
\begin{equation}
\tilde u < V_{a_0}(x) \text{ for }\{|x| \geq 1\}.
	\label{Eq:05III19-T2}
\end{equation}
To see this, first select $a_0$ sufficiently large so that
\[
\gamma(a_0) < \frac{n-2}{2}(1 + \alpha).
\]
By \eqref{Eq:05III19-T1}, there exists $R' \geq 1$ such that $V_{a_0}(x) > |x|^{-\frac{n-2}{2}(1 + \alpha)}$ for $|x| > R'$. Clearly, we can choose large $R > 1$ such that $C_0 R^{-\frac{n-2}{2}\alpha} < 1$ and $V_{a_0}(x) > C_0 R^{-\frac{n-2}{2}\alpha} |x|^{-\frac{n-2}{2}(1 + \alpha)}$ in $\{1 \leq |x| \leq R'\}$. \eqref{Eq:05III19-T2} then follows from \eqref{Eq:04III20-A3td} for these choices of $a_0$ and $R$.

\medskip
\noindent{\it Step 2:} To conclude, it suffices to show that $\tilde u \leq V_{0}$ in $\{|x| \geq 1\}$ provided $\tilde u \leq V_{a_0}$ in $\{|x| \geq 1\}$ (which holds in view of \eqref{Eq:04III20-A3td}).

Let
\[
T = \Big\{t \in [0,1]: \tilde u \leq V_{ta_0} \text{ in } \{|x| \geq 1\}\Big\} \text{ and } t_0 = \inf T.
\]
Note that $1 \in T$ due to \eqref{Eq:05III19-T2}, and so $T$ is non-empty and $t_0$ is well-defined.

We need to show that $t_0 = 0$. Suppose by contradiction that $t_0 > 0$. Note that, by \eqref{Eq:04III20-A2} and \eqref{Eq:05III19-T1},
\[
\limsup_{|x|\rightarrow \infty} \tilde u(x) |x|^{n-2} < \infty = \lim_{|x|\rightarrow \infty} V_{t_0a_0}(x) |x|^{n-2},
\]
and, by \eqref{Eq:05III19-T2},
\[
\tilde u\big|_{|x|=1}  < e^{-\frac{n-2}{2} a_0} \leq e^{-\frac{n-2}{2} t_0 a_0} = V_{t_0 a_0}(1).
\]
Thus, by the minimality of $t_0$, there exists $x_0$ with $|x_0| > 1$ such that $\tilde u (x_0) = V_{t_0a_0}(x_0)$. As $\tilde u \leq V_{t_0a_0}$, this gives a contradiction to the strong comparison principle.
We conclude that $t_0 = 0$ and so $\tilde u \leq V_{0}$ in $\{|x| > 1\}$, as desired.
\end{proof}

\section{Integral estimates and proof of Theorem \ref{Thm:20VI20-deltaEnergy}}\label{Sec:epsEnergy}

In this section, we give the proof of Theorem \ref{Thm:20VI20-deltaEnergy}. It is more convenient to use $\psi = -\frac{2}{n-2}\ln u$ so that $A^u = e^{2\psi}F[\psi]$ where 
\[
F[\psi] = \nabla^2 \psi + \nabla \psi \otimes \nabla\psi - \frac{1}{2} |\nabla \psi|^2\,I.
\]
In this section, we use $\nabla$ for the differentiation with respect to the Euclidean metric $\ringg$ on $\RR^n$. The differential relation \eqref{Eq:20VI20-X1} then becomes
\[
\sigma_k(\lambda(F[\psi])) \leq e^{-2k\psi}, \quad \lambda(F[\psi]) \in \bar\Gamma_k.
\]
Theorem \ref{Thm:20VI20-deltaEnergy} is then obtained from a priori $L^\infty$ estimate for the differential inequality
\begin{equation}
\sigma_k(\lambda(F[\psi])) \leq f \text{ and } \quad F[\psi] \in \bar\Gamma_k \text{ in } B_2
	\label{Eq:11V20-E1}
\end{equation}
where $f$ is a given function.

We prove:

\begin{proposition}\label{Prop:25V20-P1}
Let $1 \leq k \leq n/2$. Suppose $\psi \in C^2(B_2)$ satisfies \eqref{Eq:11V20-E1} for some $f \in L^p(B_2)$ with $p > \frac{n}{2k}$. For every given $q_1 >  0$ there exists $C > 0$ depending only on $n$, $p$, $q_1$ and an upper bound for $\|f\|_{L^p(B_2)}$ such that
\[
\|e^{-\psi}\|_{L^\infty(B_{1/2})} \leq C\,\|e^{-\psi}\|_{L^{q_1}(B_1)}.
\]

\end{proposition}

This result will be obtained from some Cacciopoli-type inequalities (see Propositions \ref{Prop:11V20-Cacci1}).  As in \cite{Han04, Gonzalez06-RemSing}, the proof of these results uses the Moser iteration technique.

Let us assume these results for now and give the 

\begin{proof}[Proof of Theorem \ref{Thm:20VI20-deltaEnergy}] Let $T := \sup_{\Omega} \textrm{dist}(x,\partial\Omega)^{\frac{n-2}{2}} u(x)$. It suffices to show that if $T > 4^{\frac{n-2}{2}}$, then $\int_{B_2} u^{\frac{2n}{n-2}}\,dx > \delta$ for some positive constant $\delta = \delta_n$.

Take $x_0 \in \Omega$ such that $\textrm{dist}(x,\partial\Omega)^{\frac{n-2}{2}} u(x_0)  = T$. Let $r = \frac{1}{2}\textrm{dist}(x,\partial\Omega), \lambda = u(x_0)^{\frac{2}{n-2}}$ and define
\[
\tilde u(z) = \frac{1}{u(x_0)} u(x_0 + \lambda^{-1} z) \text{ for } |z| < \lambda r.
\]
Since $T > 4^{\frac{n-2}{2}}$, we have $\lambda r > 2$. The function $\tilde u$ thus satisfies
\begin{align*}
&\sigma_k(\lambda(A^{\tilde u})) \leq 1, \quad \lambda(A^{\tilde u}) \in \Gamma_k, \quad \tilde u > 0 \text{ in } B_{2},\\
&\tilde u(0) = 1, \quad \tilde u \leq 2^{\frac{n-2}{2}} \text{ in } B_{2}.
\end{align*}

Let $\psi = -\frac{2}{n-2}\ln \tilde u$ and $f =  e^{-2k\psi}$. Then,
\begin{align*}
\sigma_k(\lambda(F[\psi]) )
	&\leq f \quad\text{ and } \quad \lambda(F[\psi]) \in \Gamma_k \text{ in } B_{2}
	,\\
\|e^{-\psi}\|_{L^n(B_1)} 
	&= \|\tilde u\|_{L^{\frac{2n}{n-2}}(B_1)}^{\frac{2}{n-2}} \leq \|u\|_{L^{\frac{2n}{n-2}}(B_2)}^{\frac{2}{n-2}} 
	,\\
\|e^{-\psi}\|_{L^\infty(B_{1/2})}
	&\geq e^{-\psi(0)} = 1.
\end{align*}
As $\|f\|_{L^\infty(B_2)} \leq C(n,k)$, we obtain from Proposition \ref{Prop:25V20-P1} that
\[
\|u\|_{L^{\frac{2n}{n-2}}(B_2)}^{\frac{2}{n-2}} \geq \|e^{-\psi}\|_{L^n(B_1)}  \geq \frac{1}{C} \|e^{-\psi}\|_{L^\infty(B_{1/2})} \geq \frac{1}{C}.
\]
The conclusion readily follows by taking $\delta < \frac{1}{C^n}$.
\end{proof}

The rest of the section is organized as follows. We start with recalling and deriving some divergence identities for the $\sigma_k$-Yamabe equation in Subsection \ref{SSec:DivId}. We then uses the identities to proved Cacciopoli-type inequalities in Subsection \ref{SSec:CaccIneql}. These will then used to prove Proposition \ref{Prop:25V20-P1}  in Subsection \ref{SSec:ProofMoser}.

\subsection{Divergence identities}\label{SSec:DivId}

Let us start some divergence identities which will be used later on. For $0 \leq \ell \leq n$, let $\NewtonT{\ell}[\psi]$ be the $\ell$-th Newton tensors of $F[\psi]$. These are defined inductively by
\begin{equation}
\NewtonT{\ell+1}[\psi] = -  \NewtonT{\ell}[\psi] F[\psi] + \sigma_{\ell+1}(F[\psi]) I, \qquad \NewtonT{0}[\psi] = I.
	\label{Eq:22IV20-NTDef}
\end{equation}
The Newton tensors satisfy the following divergence identity.
\begin{lemma}
Suppose that $0 \leq \ell \leq n$ and $\psi \in C^2(B_1)$. Then
\begin{equation} 
\nabla_a \NewtonT{\ell}[\psi]^a{}_b
	= (n - 2\ell)\NewtonT{\ell}[\psi]^a{}_b \nabla_a \psi
		- (n-\ell)\sigma_\ell(F[\psi]) \nabla_b \psi \text{ in } B_1.
		\label{Eq:22IV20-E1}
\end{equation}
\end{lemma}

Note that if we write $\tilde g = e^{-2\psi} \ringg$ and let $\tilde A$ be the $(1,1)$ Schouten tensor of $\tilde g$ and $\NewtonT{\ell}(\tilde g)$ denote its Newton tensors, then \eqref{Eq:22IV20-E1} is equivalent to the following more familiar divergence identity (see \cite{Viac00-Duke}):
\[
\nabla_a^{\tilde g} \NewtonT{\ell}(\tilde g){}^a{}_b = 0.
\]
(See \cite{LiNgGreen} for the case the metrics involved are not locally conformally flat.)
For completeness let us re-derive it here.

\begin{proof} 
Using \eqref{Eq:22IV20-NTDef}, we compute
\begin{align*}
\nabla_a \NewtonT{\ell+1}[\psi]^a{}_b
	&= -  \nabla_a \NewtonT{\ell}[\psi]^a{}_r F[\psi]^r{}_b - \NewtonT{\ell}[\psi]^a{}_r \nabla_a  F[\psi]^r{}_b + \NewtonT{\ell}[\psi]^a{}_r \nabla_b F[\psi]^r{}_a \\
	&= -  \nabla_a \NewtonT{\ell}[\psi]^a{}_r F[\psi]^r{}_b - \NewtonT{\ell}[\psi]^a{}_r [ \nabla_a  F[\psi]^r{}_b - \nabla_b F[\psi]^r{}_a ]\\
	&= -  \nabla_a \NewtonT{\ell}[\psi]^a{}_r F[\psi]^r{}_b - \NewtonT{\ell}[\psi]^a{}_r \Big[ \nabla_a  \nabla^r \psi \nabla_b \psi -  \nabla_b \nabla^r \psi \nabla_a \psi\\
		&\qquad\qquad  - \nabla_a \nabla^t \psi \nabla_t\psi \delta^r{}_b + \nabla_b \nabla^t \psi \nabla_t \psi \delta^r{}_a \Big]
\end{align*}
\begin{align*}
	&= -  \nabla_a \NewtonT{\ell}[\psi]^a{}_r F[\psi]^r{}_b - \NewtonT{\ell}[\psi]^a{}_r \Big[ F[\psi]^r{}_a \nabla_b \psi -  F[\psi]^r{}_b \nabla_a \psi \\
		&\qquad\qquad - F[\psi]^t{}_a \nabla_t\psi \delta^r{}_b + F[\psi]^t{}_b \nabla_t \psi \delta^r{}_a \Big]\\
	&= -  \nabla_a \NewtonT{\ell}[\psi]^a{}_r F[\psi]^r{}_b 
		- (\ell-1)\sigma_{\ell+1}(F[\psi]) \nabla_b \psi \\
		&\qquad
		- 2\NewtonT{\ell+1}[\psi]^a{}_b \nabla_a \psi
		- (n-\ell)\sigma_\ell(F[\psi]) F[\psi]^t{}_b \nabla_t \psi .
\end{align*}
A straightforward induction on $\ell$ using the above recursive formula yields the desired conclusion.
\end{proof}

\begin{corollary}
Suppose $0 \leq \ell \leq n - 1$, $p \geq 0$, $\psi \in C^2(B_1)$. For $q \in \RR$, it holds that
\begin{align}
&\nabla_a(e^{-q\psi} |\nabla \psi|^p\NewtonT{\ell}[\psi]^a{}_b\nabla^b \psi )
\nonumber\\
	&\quad =  -p e^{-q\psi}  |\nabla \psi|^{p-2} \NewtonT{\ell+1}[\psi]^c{}_b \nabla^b \psi \nabla_c \psi + (p+\ell+1) e^{-q\psi}  \sigma_{\ell+1}(F[\psi]) |\nabla \psi|^p
	\nonumber\\
		&\qquad
		+   (n - 2\ell - q - \frac{p}{2} - 1) e^{-q\psi} |\nabla \psi|^p  \NewtonT{\ell}[\psi]^a{}_b \nabla_a \psi \nabla^b \psi \nonumber\\
		&\qquad 
		- \frac{n-\ell}{2} e^{-q\psi}  \sigma_\ell (F[\psi]) |\nabla\psi|^{p+2} \text{ in } B_1.
		\label{Eq:22IV20-E2} 
\end{align}
\end{corollary}

\begin{proof} It suffices to consider the case $q = 0$, since the general case follows from this case immediately. We compute
\begin{align*}
&(\ell+1)\sigma_{\ell+1}(F[\psi])  |\nabla\psi|^p
		= |\nabla \psi|^p\NewtonT{\ell}[\psi]^a{}_bF[\psi]^b{}_a  \\
	&\quad= |\nabla \psi|^p\NewtonT{\ell}[\psi]^a{}_b(\nabla^b\nabla_a \psi + \nabla^b \psi \nabla_a\psi - \frac{1}{2} |\nabla \psi|^2\,\delta^b{}_a)
		\\
	&\quad \stackrel{\eqref{Eq:22IV20-E1}}{=} \nabla_a(|\nabla \psi|^p\NewtonT{\ell}[\psi]^a{}_b\nabla^b \psi)  \\
		&\qquad
		- |\nabla \psi|^p   \Big[ (n - 2\ell)\NewtonT{\ell}[\psi]^a{}_b \nabla_a \psi
		- (n-\ell)\sigma_\ell(F[\psi]) \nabla_b \psi \Big] \nabla^b \psi
		\\
		&\qquad
		- p |\nabla \psi|^{p-2} \NewtonT{\ell}[\psi]^a{}_b  \Big[F[\psi]^c{}_a -  \nabla_a \psi \nabla^c \psi + \frac{1}{2}|\nabla\psi|^2 \delta^c{}_a \Big]\nabla^b \psi \nabla_c \psi
		\\
		&\qquad
		+  |\nabla \psi|^p\NewtonT{\ell}[\psi]^a{}_b \nabla_a \psi \nabla^b \psi - \frac{n-\ell}{2} \sigma_\ell (F[\psi]) |\nabla\psi|^{p+2}\\
	&\quad \stackrel{\eqref{Eq:22IV20-NTDef}}{=} \nabla_a(|\nabla \psi|^p\NewtonT{\ell}[\psi]^a{}_b\nabla^b \psi)  \\
		&\qquad
		-   (n - 2\ell - \frac{p}{2} - 1) |\nabla \psi|^p  \NewtonT{\ell}[\psi]^a{}_b \nabla_a \psi \nabla^b \psi  
		+ \frac{n-\ell}{2} \sigma_\ell (F[\psi]) |\nabla\psi|^{p+2}
		\\
		&\qquad
		+ p |\nabla \psi|^{p-2} \NewtonT{\ell+1}[\psi]^c{}_b \nabla^b \psi \nabla_c \psi - p \sigma_{\ell+1}(F[\psi]) |\nabla \psi|^p.
\end{align*}
Rearranging terms, we obtain identity \eqref{Eq:22IV20-E2} when $q = 0$.
\end{proof}

\begin{corollary}\label{Cor:23IV20-C1}
Suppose $\psi \in C^2(B_1)$. For $q \in \RR$ and $t,s > 0$, it holds that
\begin{align*}
&\nabla_a\Big\{e^{-q\psi}\Big[\sum_{j=0}^{k-1} \frac{t^{(j)}}{2^j s^{(j)} }  |\nabla \psi|^{2j}\NewtonT{k-1-j}[\psi]^a{}_b
	\Big]\nabla^b \psi \Big\}
	\nonumber\\
		&\quad =  k  e^{-q\psi}\sigma_k(F[\psi]) 
		\nonumber\\
		&\qquad
		+ e^{-q\psi} \sum_{j=0}^{k-1}  \frac{t^{(j)}}{2^j s^{(j)} } \Big(n - 2k + \frac{(j+1)(s-t)}{s + j} - q \Big)  
			|\nabla \psi|^{2j} \NewtonT{k-1-j}[\psi]^a{}_b \nabla_a \psi \nabla^b \psi
		\nonumber\\
		&\qquad 
		- e^{-q\psi} \sum_{j=0}^{k-1}  \frac{t^{(j)}}{2^{j+1} s^{(j+1)} } \Big((n-k+1)s - (k+1)t + (n-2k + s - t)j\Big)  \times
		\nonumber\\
			&\qquad\qquad \times
			|\nabla \psi|^{2j + 2} \sigma_{k-1-j}(F[\psi]) \text{ in } B_1.
\end{align*}
Here we have used the rising factorial notation $x^{(j)}$, i.e.
\[
x^{(0)} = 1 \text{ and } x^{(j)} = \prod_{i = 0}^{j-1} (x + i) \text{ for } j \geq 1.
\]
\end{corollary}

\begin{proof} Applying \eqref{Eq:22IV20-E2}, we have the following identities:
\begin{align*}
&\nabla_a(e^{-q\psi}\NewtonT{k-1}[\psi]^a{}_b\nabla^b \psi) 
\\
	&\quad=  k e^{-q\psi}\sigma_k(F[\psi]) \\
		&\qquad
		+   (n - 2k  + 1 - q)  e^{-q\psi} \NewtonT{k-1}[\psi]^a{}_b \nabla_a \psi \nabla^b \psi  
		- \frac{n-k+1}{2} e^{-q\psi}\sigma_{k-1} (F[\psi]) |\nabla\psi|^{2}
		,\\
&\nabla_a(e^{-q\psi}|\nabla \psi|^2\NewtonT{k-2}[\psi]^a{}_b\nabla^b \psi) 
\\
	&\quad =  -2 e^{-q\psi}\NewtonT{k-1}[\psi]^c{}_b \nabla^b \psi \nabla_c \psi 
		+ (k+1) e^{-q\psi}\sigma_{k-1}(F[\psi]) |\nabla \psi|^2\\
		&\qquad
		+   (n - 2k + 2 - q)e^{-q\psi} |\nabla \psi|^2  \NewtonT{k-2}[\psi]^a{}_b \nabla_a \psi \nabla^b \psi  
		- \frac{n-k+2}{2} e^{-q\psi}\sigma_{k-2} (F[\psi]) |\nabla\psi|^{4}
		,\\
&\nabla_a(e^{-q\psi}|\nabla \psi|^4\NewtonT{k-3}[\psi]^a{}_b\nabla^b \psi) 
\\
	&\quad=  -4 e^{-q\psi}|\nabla \psi|^{2} \NewtonT{k-2}[\psi]^c{}_b \nabla^b \psi \nabla_c \psi
		 + (k+2)e^{-q\psi} \sigma_{k-2}(F[\psi]) |\nabla \psi|^4\\
		&\qquad
		+   (n  -2k + 3 - q) e^{-q\psi}|\nabla \psi|^4  \NewtonT{k-3}[\psi]^a{}_b \nabla_a \psi \nabla^b \psi  
		- \frac{n-k+3}{2} e^{-q\psi}\sigma_{k-3} (F[\psi]) |\nabla\psi|^{6}
		,\\
	&\ldots
\end{align*}
Multiplying the first identity with $1$, the second identity by $\frac{t}{2s}$, the third identity by $\frac{t}{2s} \frac{t+1}{2(s+1)}$, and so on, and adding the resulting identities together, we obtain the conclusion.
\end{proof}

Choosing $t = n - k + 1$ and $s = k + 1 + \delta$ in Corollary \ref{Cor:23IV20-C1}, we obtain
\begin{corollary}\label{Cor:23IV20-C2}
Suppose $\psi \in C^2(B_1)$. For $q \in \RR$ and $\delta > - k - 1$, it holds that
\begin{align*}
&\nabla_a\Big\{e^{-q\psi}\Big[\sum_{j=0}^{k-1} \frac{(n - k + 1)^{(j)}}{2^j (k + 1 + \delta)^{(j)} }  |\nabla \psi|^{2j}\NewtonT{k-1-j}[\psi]^a{}_b
	\Big]\nabla^b \psi \Big\}
	\nonumber\\
		&\quad =  k  e^{-q\psi}\sigma_k(F[\psi]) 
		\nonumber\\
		&\qquad
		+ e^{-q\psi} \sum_{j=0}^{k-1}  \frac{(n-k+1)^{(j)}}{2^j (k+1+\delta)^{(j)} } \Big(\frac{k(n-2k) + \delta(n - 2k + 1 + j)}{k + 1 + \delta  + j} - q\Big) \times  \nonumber\\
			&\qquad\qquad \times
			|\nabla \psi|^{2j} \NewtonT{k-1-j}[\psi]^a{}_b \nabla_a \psi \nabla^b \psi
		\nonumber\\
		&\qquad 
		- \delta e^{-q\psi} \sum_{j=0}^{k-1}  \frac{(n-k+1)^{(j + 1)}}{2^{j+1} (k+1+\delta)^{(j+1)} }
			|\nabla \psi|^{2j + 2} \sigma_{k-1-j}(F[\psi]) \text{ in } B_1.
\end{align*}
\end{corollary}

\subsection{Cacciopoli-type inequalities for $\sigma_k$ equations}
\label{SSec:CaccIneql}

In view of Corollary \ref{Cor:23IV20-C2}, for `large' $q$, we can recast \eqref{Eq:11V20-E1} in the form
\[
	\nabla_a \big(e^{-q\psi} X^a{}_b\nabla^b \psi \big) \leq k  e^{-q\psi}f  - \text{ (positive) good terms}
\]
where the matrix $X^a{}_b$ is non-negative definite and the good terms on the right hand side are are multiples of $e^{-q\psi} |\nabla \psi|^{2} \sigma_{k-1}(F[\psi])$, \ldots, $e^{-q\psi} |\nabla \psi|^{2k} \sigma_{0}(F[\psi])$. (See \eqref{Eq:23IV20-R1} below for the exact expression.) Exploiting this will lead us to the following Cacciopoli-type inequality.

\begin{proposition}[Cacciopoli-type inequality]\label{Prop:11V20-Cacci1}
Suppose $1 \leq k \leq n$, $\theta > 0$, $q > \theta + \frac{k(n-2k)}{k+1}$, $\psi \in C^2(B_1)$ and $F[\psi] \in \bar\Gamma_{k}$ in $B_1$. There exist $C = C(n,\theta) > 0$ and $\alpha = \alpha(n) > 0$ such that, for $0 < r < R < 1$,
\[
\int_{B_r} e^{-q\psi} |\nabla \psi|^{2k}\,dx
		\leq  C \int_{B_R}   e^{-q\psi}\sigma_k(F[\psi]) \,dx
		 + \frac{(q + C)^{\alpha} }{(R - r)^{2k}} \int_{B_R}  e^{-q\psi} \,dx.
\]
\end{proposition}

In a similar vein, but by working with `large' negative $q$, we get:

\begin{proposition}[Cacciopoli-type inequality]\label{Prop:11V20-Cacci2}
Suppose $1 \leq k \leq n$, $\theta > 0$, $s > \theta + \frac{k(2k-n)}{k+1}$, $\psi \in C^2(B_1)$ and $F[\psi] \in \bar\Gamma_{k}$ in $B_1$. There exist $C = C(n,\theta) > 0$ and $\alpha = \alpha(n) > 0$ such that, for $0 < r < R < 1$,
\[
\int_{B_r} e^{s\psi} |\nabla \psi|^{2k}\,dx
		\leq  \frac{(s + C)^{\alpha} }{(R - r)^{2k}} \int_{B_R}  e^{s\psi} \,dx.
\]
\end{proposition}

We will use the following consequence of the identity  \eqref{Eq:22IV20-E2}.

\begin{lemma}\label{Lem:22IV20-L1}
Suppose $0 \leq \ell \leq n-1$, $p \geq 0$, $q \in \RR$, $\psi \in C^2(B_1)$ and $F[\psi] \in \bar\Gamma_{\ell + 1}$ in $B_1$. For every cut-off function $\eta \in C_c^\infty(B_1)$, we have
\begin{align*}
&\Big| (p + \ell + 1) \int_{B_1} \eta^{2(\ell + 1)} e^{-q\psi} |\nabla \psi|^p \sigma_{\ell + 1}(F[\psi])\,dx\\
	&\qquad 
	- p \int_{B_1} \eta^{2(\ell + 1)} e^{-q\psi}  |\nabla \psi|^{p-2} \NewtonT{\ell + 1}[\psi]^a{}_b \nabla^b \psi \nabla_a \psi\,dx\Big|\\
	&\qquad\qquad
	\leq (n-\ell)(|q| + \frac{p}{2} + C_n) \int_{B_1} \eta^{2\ell + 2} e^{-q\psi} |\nabla \psi|^{p+2} \sigma_{\ell}(F[\psi]) \,dx\\
		&\quad\qquad\qquad  + C_n \int_{B_1} \eta^{2\ell} e^{-q\psi}  |\nabla \eta|^2 |\nabla \psi|^{p} \sigma_{\ell}(F[\psi]) \,dx.
\end{align*}
\end{lemma}

\begin{proof} Multiply \eqref{Eq:22IV20-E2} by $\eta^{2(\ell + 1)}$ and integrate over $B_1$ we obtain
\begin{align*}
&
\int_{B_1} \eta^{2(\ell + 1)} e^{-q\psi} \Big[(p + \ell + 1) |\nabla \psi|^p \sigma_{\ell + 1}(F[\psi]) - p|\nabla \psi|^{p-2} \NewtonT{\ell + 1}[\psi]^a{}_b \nabla^b \psi \nabla_a \psi\Big]\,dx\\
	&= - 2(\ell + 1)\int_{B_1}  \eta^{2\ell +1} e^{-q\psi} |\nabla \psi|^p \NewtonT{\ell}[\psi]^a{}_b\nabla^b \psi \nabla_a \eta\,dx
	\\
		&\qquad - (n - 2\ell - q - \frac{p}{2} - 1) \int_{B_1} \eta^{2(\ell + 1)}  e^{-q\psi} |\nabla \psi|^p \NewtonT{\ell}[\psi]^a{}_b \nabla_a \psi \nabla^b \psi  \,dx\\
		&\qquad 
		+ \frac{n-\ell}{2} \int_{B_1} \eta^{2(\ell + 1)} e^{-q\psi}   |\nabla\psi|^{p + 2}\sigma_{\ell} (F[\psi])\,dx.
\end{align*}
Noting that $\NewtonT{\ell}[\psi]$ is non-negative definite for $F[\psi] \in \bar\Gamma_{\ell+1}$ and $\tr(\NewtonT{\ell}[\psi]) = (n-\ell)\sigma_\ell(F[\psi])$, we have 
\[
\Big|\eta \sum_{a,b} \NewtonT{\ell}[\psi]^a{}_b\nabla^b \psi \nabla_a \eta\Big| \leq C\sigma_\ell(F[\psi]) (\eta^2 |\nabla \psi|^2 + |\nabla \eta|^2).
\]
Using this to estimate the first integral on the right hand side of the previous identity, we arrive at the conclusion.
\end{proof}

\begin{proof}[Proof of Proposition \ref{Prop:11V20-Cacci1}] We use $C$ to denote a generic positive constant depending only on $n$, $k$ and $\theta$. In our argument, we use the fact that $\NewtonT{\ell}[\psi]$ is non-negative definite for $0 \leq \ell \leq k -1$ whenever $F[\psi] \in \bar\Gamma_k$.  By Corollary \ref{Cor:23IV20-C2}, we can select a sufficiently small $\delta = \delta(\theta, n, k)$ such that for $q > \theta + \frac{k(n-2k)}{k+1}$ and $F[\psi] \in \bar\Gamma_k$, 
\begin{align}
&\nabla_a\Big\{e^{-q\psi}\Big[\sum_{j=0}^{k-1} \frac{(n - k + 1)^{(j)}}{2^j (k + 1 + \delta)^{(j)} }  |\nabla \psi|^{2j}\NewtonT{k-1-j}[\psi]^a{}_b
	\Big]\nabla^b \psi \Big\}
	\nonumber\\
		&\quad \leq  k  e^{-q\psi}\sigma_k(F[\psi]) 
		\nonumber\\
		&\qquad 
		- \delta e^{-q\psi} \sum_{j=0}^{k-1}  \frac{(n-k+1)^{(j + 1)}}{2^{j+1} (k+1+\delta)^{(j+1)} }
			|\nabla \psi|^{2j + 2} \sigma_{k-1-j}(F[\psi]).
	\label{Eq:23IV20-R1}
\end{align}

Let $\eta \in C_c^\infty(B_R)$ be a cut-off function such that $0 \leq \eta \leq 1$ in $B_R$, $\eta \equiv 1$ in $B_r$ and $|\nabla \eta| \leq \frac{C}{R - r}$. Multiplying \eqref{Eq:23IV20-R1} by $\eta^{2k}$ and integrating by parts we obtain
\begin{align}
 &\frac{1}{C} \int_{B_R} \eta^{2k} e^{-q\psi} \sum_{j=0}^{k-1}   
			|\nabla \psi|^{2j + 2} \sigma_{k-1-j}(F[\psi])\,dx
		\nonumber\\
	&\quad\leq  k \int_{B_R} \eta^{2k}  e^{-q\psi}\sigma_k(F[\psi]) \,dx\nonumber\\
		&\qquad + 2k\int_{B_R} \eta^{2k-1} \,e^{-q\psi} \sum_{j=0}^{k-1}  \frac{(n - k + 1)^{(j)}}{2^j (k + 1 + \delta)^{(j)} } |\nabla \psi|^{2j}\NewtonT{k-1-j}[\psi]^a{}_b \nabla^b \psi \nabla_a \eta\,dx.
		\label{Eq:26V20-A1}
\end{align}
Noting that $\NewtonT{k-1-j}[\psi]$ is non-negative definite and its trace is a multiple of $\sigma_{k-1-j}(F[\psi])$, we have
\[
\Big|\sum_{a,b} \NewtonT{k-1-j}[\psi]^a{}_b \nabla^b \psi \nabla_a \eta\Big|
	\leq C\sigma_{k-1-j}(F[\psi])|\nabla \psi||\nabla\eta|.
\]
It follows that
\begin{align*}
& \int_{B_R} \eta^{2k} e^{-q\psi} \sum_{j=0}^{k-1}   
			|\nabla \psi|^{2j + 2} \sigma_{k-1-j}(F[\psi])\,dx
		\\
	&\quad\leq  C \int_{B_R} \eta^{2k}  e^{-q\psi}\sigma_k(F[\psi]) \,dx\\
		&\qquad + C\int_{B_R} \eta^{2k-1} \,e^{-q\psi}\sum_{j=0}^{k-1}  |\nabla \eta| |\nabla \psi|^{2j + 1} \sigma_{k-1-j}(F[\psi])\,dx.
\end{align*}
Using the inequality
\[
\eta^{-1} |\nabla \eta| |\nabla \psi|^{2j+1} \leq \varepsilon |\nabla \psi|^{2j+2} + \frac{C}{\varepsilon^{2j+1}}\eta^{-2j-2} |\nabla\eta|^{2j+2} \text{ for all } \varepsilon > 0,
\]
we are led to
\begin{align}
& \int_{B_R} \eta^{2k} e^{-q\psi} \sum_{j=0}^{k-1}   
			|\nabla \psi|^{2j + 2} \sigma_{k-1-j}(F[\psi])\,dx
		\nonumber\\
	&\quad\leq  C \int_{B_R} \eta^{2k}  e^{-q\psi}\sigma_k(F[\psi]) \,dx
	\nonumber\\
		&\qquad + C \int_{B_R}  e^{-q\psi}\sum_{j=0}^{k-1} \frac{1}{(R - r)^{2j + 2}} \eta^{2(k - 1 - j)}  \sigma_{k-1-j}(F[\psi])\,dx.
		\label{Eq:23IV20-R2}
\end{align}

Consider the last integral on the right hand side of \eqref{Eq:23IV20-R2}. When $j \leq k - 2$, we use Lemma \ref{Lem:22IV20-L1} with $p = 0$, $\ell = k - 2 - j$ to obtain
\begin{align*}
& \frac{k - 1 - j}{n-k+2+j} \int_{B_R} \eta^{2(k - 1 - j)} e^{-q\psi} \sigma_{k - 1 - j}(F[\psi])\,dx\\
	&\qquad
	\leq (q + C) \int_{B_R} \eta^{2(k-1-j)} e^{-q\psi} |\nabla \psi|^{2} \sigma_{k-2-j}(F[\psi]) \,dx\\
		&\quad\qquad  + C \int_{B_R} \eta^{2(k-2-j)} e^{-q\psi} |\nabla \eta|^2  \sigma_{k-2-j}(F[\psi]) \,dx.
\end{align*}
Using the inequality
\[
\eta^{-2(1+j)} |\nabla \psi|^2 \leq \frac{\varepsilon (R - r)^{2j + 2}}{(q + C)^{1+\alpha}}  |\nabla\psi|^{2j + 4} + \frac{C  (q+C)^{\frac{1 + \alpha}{1+j}}}{ \varepsilon^{\frac{1}{j+1}}(R - r)^2} \eta^{-2(2+j)} \text{ for all } \alpha \geq 0, \varepsilon > 0
\]
in the first integral on the right hand side of the previous inequality, we obtain
\begin{align}
&\frac{(q + C)^{\alpha}}{ (R - r)^{2j+2}} \int_{B_R} \eta^{2(k - 1 - j)} e^{-q\psi} \sigma_{k - 1 - j}(F[\psi])\,dx
	\nonumber\\
	&\qquad
	\leq \varepsilon  \int_{B_R} \eta^{2k} e^{-q\psi} |\nabla \psi|^{2j + 4} \sigma_{k-2-j}(F[\psi]) \,dx
	\nonumber\\
		&\quad\qquad  + \frac{C  (q+C)^{\frac{(1+\alpha)(2+j)}{1+j}}}{ \varepsilon^{\frac{1}{j+1}}(R - r)^{2j+4}} \int_{B_R} \eta^{2(k-2-j)} e^{-q\psi}  \sigma_{k-2-j}(F[\psi]) \,dx.
	\label{Eq:24IV20-R3}
\end{align}

Using \eqref{Eq:24IV20-R3} repeatedly for $j = 0$, $j = 1$, \ldots, $j = k-2$ to treat the last integral on the right hand side of \eqref{Eq:23IV20-R2}, we obtain for some $\alpha = \alpha(n,k) > 0$ that
\begin{align*}
& \int_{B_R} \eta^{2k} e^{-q\psi} \sum_{j=0}^{k-1}   
			|\nabla \psi|^{2j + 2} \sigma_{k-1-j}(F[\psi])\,dx
		\nonumber\\
	&\quad\leq  C \int_{B_R} \eta^{2k}  e^{-q\psi}\sigma_k(F[\psi]) \,dx
		 + \frac{(q + C)^{\alpha} }{(R - r)^{2k}} \int_{B_R}  e^{-q\psi} \,dx.
\end{align*}
The conclusion follows.
\end{proof}

\begin{proof}[Proof of Proposition \ref{Prop:11V20-Cacci2}] We adapt the proof of Proposition \ref{Prop:11V20-Cacci1}. We use $C$ to denote a generic positive constant depending only on $n$, $k$ and $\theta$. Appying Corollary \ref{Cor:23IV20-C2}, we can select a sufficiently small $\delta = \delta(\theta, n, k) \in (0,1)$ such that for $s = -q > \frac{k(2k-n)}{k+1} + \theta$ and $F[\psi] \in \bar\Gamma_k$, 
\begin{align}
&\nabla_a\Big\{e^{s\psi}\Big[\sum_{j=0}^{k-1} \frac{(n - k + 1)^{(j)}}{2^j (k + 1 - \delta)^{(j)} }  |\nabla \psi|^{2j}\NewtonT{k-1-j}[\psi]^a{}_b
	\Big]\nabla^b \psi \Big\}
	\nonumber\\
		&\quad \geq \delta e^{s\psi} \sum_{j=0}^{k-1}  \frac{(n-k+1)^{(j + 1)}}{2^{j+1} (k+1-\delta)^{(j+1)} }
			|\nabla \psi|^{2j + 2} \sigma_{k-1-j}(F[\psi]).
	\label{Eq:26V20-R1}
\end{align}

As before, we let $\eta \in C_c^\infty(B_R)$ be a cut-off function such that $0 \leq \eta \leq 1$ in $B_R$, $\eta \equiv 1$ in $B_r$ and $|\nabla \eta| \leq \frac{C}{R - r}$. Multiplying \eqref{Eq:26V20-R1} by $\eta^{2k}$ and integrating we obtain (compare \eqref{Eq:26V20-A1})
\begin{align*}
 &\frac{1}{C} \int_{B_R} \eta^{2k} e^{s\psi} \sum_{j=0}^{k-1}   
			|\nabla \psi|^{2j + 2} \sigma_{k-1-j}(F[\psi])\,dx
		\\
	&\quad\leq - 2k\int_{B_R} \eta^{2k-1} \,e^{s\psi} \sum_{j=0}^{k-1}  \frac{(n - k + 1)^{(j)}}{2^j (k + 1 - \delta)^{(j)} } |\nabla \psi|^{2j}\NewtonT{k-1-j}[\psi]^a{}_b \nabla^b \psi \nabla_a \eta\,dx.
\end{align*}
We can now argue as in exactly as in the proof of Proposition \ref{Prop:11V20-Cacci1} to obtain the conclusion.
\end{proof}

\subsection{Proof of Propositions \ref{Prop:25V20-P1}}
\label{SSec:ProofMoser}

\begin{proof}[Proof of Proposition \ref{Prop:25V20-P1}] We use Moser's iteration. It suffices to give the proof for some $q_1 > 0$. The fact that the result holds for all $q_1 > 0$ follows by interpolation.

Fix some $\theta >0$ and $q_1 > \theta +  \frac{k(n-2k)}{k+1}$. Let $\chi = \frac{n}{n-2k} > \frac{p}{p-1}$ if $1 < k < n/2$ and $\chi > \frac{p}{p-1}$ be arbitrary if $k = n/2$. Below we use $C$ to denote some positive constant depending only on $n$, $k$, $\theta$ and $\chi$.

By the Cacciopoli-type inequality in Proposition \ref{Prop:11V20-Cacci1}, there exists $\alpha = \alpha(n,k,\theta) > 0$ such that for all $q \geq q_1$ and $1/2 < r < R < 1$, 
\begin{equation}
\int_{B_r} e^{-q\psi} |\nabla \psi|^{2k}\,dx
		\leq  C \int_{B_{R}}   e^{-q\psi} f \,dx
		 + \frac{Cq^\alpha}{(R - r)^{2k}} \int_{B_R}  e^{-q\psi} \,dx.
	\label{Eq:03VIII20-A1}
\end{equation}
Applying $W^{1,2k}$-Sobolev's inequality to $e^{-\frac{q}{2k}\psi}$, we thus have
\begin{eqnarray*}
\Big(\int_{B_r}  e^{-q\chi \psi}\,dx\Big)^{1/\chi}
	&\leq& C\int_{B_r} [q^{2k}e^{-q\psi} |\nabla \psi|^{2k} + e^{-q\psi}]\,dx\\
	&\stackrel{\eqref{Eq:03VIII20-A1}}{\leq}& Cq^{2k} \int_{B_{R}}   e^{-q\psi} f \,dx
		 + \frac{Cq^{\alpha + 2k}}{(R - r)^{2k}} \int_{B_R}  e^{-q\psi} \,dx.
\end{eqnarray*}
On the other hand, by H\"older's and Young's inequalities,
\begin{align*}
\int_{B_{R}}   e^{-q\psi} f \,dx 
	&= \int_{B_{R}}   e^{-\frac{\chi}{(\chi-1)p}q\psi} e^{-(1 - \frac{1}{p}  - \frac{1}{(\chi-1)p})q\psi} f \,dx\\
	&\leq \Big(\int_{B_{R}}   e^{-\chi q\psi}\,dx\Big)^{\frac{1}{(\chi-1)p}} \Big(\int_{B_R} e^{-q\psi} \,dx\Big)^{\frac{\chi p - \chi - p}{(\chi-1)p}} \Big( \int_{B_R} f^p \,dx\Big)^{1/p}\\
	&\leq \varepsilon \Big(\int_{B_{R}}   e^{- \chi q\psi}\,dx\Big)^{1/\chi} + C\varepsilon^{-\frac{(\chi - 1)p}{\chi p - \chi  - p}}   \Big(\int_{B_R} e^{-q\psi}  \,dx\Big)\Big( \int_{B_R} f^p \,dx\Big)^{\frac{\chi-1}{\chi p - \chi - p}}.
\end{align*}
Therefore we have for some $\beta = \beta(n,k,\theta,\chi) > 0$ that
\begin{align*}
\|e^{-\psi}\|_{L^{q\chi}(B_r)}^{q} 
	&\leq \frac{1}{2} \|e^{-\psi}\|_{L^{q\chi}(B_R)}^{q} 
		 + C q^\beta\Big(\frac{1}{(R - r)^{2k}} + \|f\|_{L^p(B_1)}^{\frac{(\chi-1)p}{\chi p - \chi - p}} \Big) 
		\| e^{-\psi}\|_{L^q(B_R)}^q.
\end{align*}
A standard iteration (see \cite[Lemma 1.1]{GiaquintaGiusti}) then gives 
\begin{equation}
\|e^{-\psi}\|_{L^{q\chi}(B_r)}^{q} 
	\leq C q^\beta\Big(\frac{1}{(R - r)^{2k}} + \|f\|_{L^p(B_1)}^{\frac{(\chi-1)p}{\chi p - \chi - p}} \Big) \| e^{-\psi}\|_{L^q(B_R)}^q
	\label{Eq:05VIII20-A1}
\end{equation}
for all $1/2 < r < R < 1$ and $q \geq q_1$. 

Now, consider $r_j = \frac{1}{2} + \frac{1}{2^j}$ and $q_j = q_1 \,\chi^{j-1}$ for $j \geq 1$. By \eqref{Eq:05VIII20-A1}, we have
\begin{align*}
\|e^{-\psi}\|_{L^{q_{j+1}}(B_{r_{j+1}})}
	\leq  C^{q_1^{-1} \chi^{-j}} \chi^{\beta\,q_1^{-1} j\,\chi^{-j}} \Big(2^{2jk} + \|f\|_{L^p(B_1)}^{\frac{(\chi-1)p}{\chi p - \chi - p}}\Big)^{q_1^{-1}\chi^{-j}} \|e^{-\psi}\|_{L^{q_j}(B_{r_j})}.
\end{align*}
It follows that
\begin{align*}
\|e^{-\psi}\|_{L^{q_{m+1}}(B_{r_{m+1}})}
	&\leq  \prod_{j=1}^m \Big\{C^{q_1^{-1} \chi^{-j}} \chi^{\beta\,q_1^{-1} j\,\chi^{-j}} \Big(2^{2jk} + \|f\|_{L^p(B_1)}^{\frac{(\chi-1)p}{\chi p - \chi - p}}\Big)^{q_1^{-1}\chi^{-j}}\Big\} \|e^{-\psi}\|_{L^{q_1}(B_{r_1})}\\
	&\leq C(n,k,\theta,\chi, \|f\|_{L^p(B_1)}) \|e^{-\psi}\|_{L^{q_1}(B_{r_1})} \text{ for all } m \geq 1.
\end{align*}
Sending $m \rightarrow \infty$ we obtain the conclusion.
\end{proof}

\section{Proof of Theorem \ref{main'}}\label{Sec:ProofMain}

Fix some $0 < \alpha' \leq \alpha < 1$. For $\ell \geq 0$, let $C^{\ell,\alpha}_+(\SSphere^n)$ (resp. $C^{\ell}_+(\SSphere^n)$)  denote the cone of positive functions in $C^{\ell,\alpha}(\SSphere^n)$ (resp. $C^\ell(\SSphere^n)$). 

Consider first the case $K \in C^{2,\alpha}_+(\SSphere^n)$ satisfying the non-degeneracy condition \eqref{nondegeneracy}. The general case where $K \in C^2_+(\SSphere^n)$ will be dealt with via approximation. 

For any $\mu\in[0,1]$, let $K_\mu = \mu K+(1-\mu)2^{-k} \Big(\begin{array}{c}n\\k\end{array}\Big)$ and consider the equation
\begin{equation}
\sigma_{k}(\lambda(A_{g_{v}})) = K_{\mu}, \quad\lambda(A_{g_v})\in\Gamma_{k} \text{ on } \SSphere^n.
\label{Eq:27VII20-Xmu}
\end{equation}
By Theorem \ref{thm9999} and first and second derivative estimates for the $\sigma_k$-Yamabe equation (see \cite{Chen05, GW03-IMRN}, \cite[Theorem 1.10]{Li09-CPAM}, \cite[Theorem 1.20]{LiLi03}, \cite{Wang06}), we can select $C_*$ sufficiently large such that, for $\mu \in (0,1]$, all positive solutions to \eqref{Eq:27VII20-Xmu} belong to the set
\begin{equation}
\mcO = \Big\{\tilde v \in C^{4,\alpha'}(\SSphere^n): \|\ln \tilde v\|_{C^{4,\alpha'}(\SSphere^n)} < \pi^{\alpha-\alpha'} C_*, \lambda(A_{g_{\tilde v}}) \in \Gamma_k\Big\}.
	\label{Eq:25VI20-ODef}
\end{equation}

Consider the nonlinear operator $F_{\mu}: \mcO \rightarrow C^{2,\alpha'}(\SSphere^{n})$ defined by
\begin{equation}
F_{\mu}[v]:=\sigma_{k}(\lambda(A_{g_{v}}))-K_{\mu},\quad\forall~v\in \mcO .
	\label{Eq:25VI20-FmuDef}
\end{equation}
By \cite{Li89-CPDE}, the degree $\deg(F_\mu, \mcO, 0)$ is well-defined and is independent of $\mu \in (0,1]$. This degree is also independent of $\alpha' \in (0,\alpha]$ (see \cite[Theorem B.1]{Li95-JDE}). We proceed to compute this degree for small $\mu$ and some $\alpha' \in (0,\alpha)$.

Our computation closely follows \cite{Li95-JDE}, using the Lyapunov-Schmidt reduction. Let us briefly give an outline. We start by parametrizing  $C^{4,\alpha'}_+(\SSphere^n)$ as $\pi(\mcS_0 \times B)$ for some parametrization $\pi$ (see Lemma \ref{Lem:24VI20-L1}) where the $B$-factor takes into account the the action of the M\"obius group on $\SSphere^n$ and where the element $1 \in \mcS_0$ corresponds to the so-called standard bubbles on $\SSphere^n$. An important property of this parametrization is that, for every given tubular neighborhood $\pi(\mcN \times B)$ of $\pi(\{1\} \times B)$, all solutions of \eqref{Eq:27VII20-Xmu} belongs to $\pi(\mcN \times B)$ if $\mu$ is sufficiently small; see Lemma \ref{Lem:25VI20-wConv}. Thus, by the excision property, it suffices to compute $\deg(F_\mu,\pi(\mcN \times B), 0)$. We then appeal to the implicit function theorem to show that, for every given $\xi \in B$, there is a unique $w_{\xi,\mu} \in \mcN$ such that the `$\mcS_0$-component' of $F_\mu(\pi(w_{\xi,\mu},\xi))$ is in fact zero; see Proposition \ref{prop233}. The problem of solving for zeroes of $F_\mu$ in $\pi(\mcN \times B)$ thus reduces to solve for zeroes of certain finite dimensional map $\xi \mapsto \Lambda_{\xi,\mu}$, whose degree was computed in \cite{Li95-JDE}. Finally, we complete the argument by showing $(-1)^n\deg(F_\mu,\pi(\mcN \times B), 0)$ is the degree of $\Lambda_{\xi,\mu}$; see Theorem \ref{Thm:dff}.

Let us now give the details, starting with the declared parametrization of $C^{4,\alpha'}_+(\SSphere^n)$. For $P\in\SSphere^{n}$ and $1\leq t<\infty$, we define a conformal transformation $\varphi_{P,t}$ on $\SSphere^{n}$ 
by sending $y$ to $ty$, where $y$ is the stereographic projection coordinates of points while the stereographic projection is performed with $P$ as the north pole to the equatorial plane of $\SSphere^{n}$.

For a conformal transformation $\varphi:\SSphere^{n}\rightarrow \SSphere^{n}$ as above and a function $v$ defined on $\SSphere^n$, we let 
\[
T_{\varphi}v:=v\circ\varphi|\det d\varphi|^{\frac{n-2}{2n}}
\]
where $d\varphi$ denotes the Jacobian of $\varphi$. In particular, the pull-back metric of $g_v = v^{\frac{4}{n-2}}\roundg$ under $\varphi$ is given by $\varphi^* (g_v) = g_{T_\varphi v}$.

Let $B$ denote the open unit ball in $\RR^{n+1}$ and let
\begin{equation*}
\mcS_{0} =\Big\{v\in C^{4,\alpha'}_+(\SSphere^{n}):\int_{\SSphere^{n}}x |v(x)|^{\frac{2n}{n-2}}\,dv_{\roundg}(x) =0\Big\}.
\end{equation*}
Note that our definition of $\mcS_0$ differs slightly from that in \cite{Li95-JDE} in that we do not require $g_v$ to have unit volume.

For $w \in \mcS_0$ and $\xi \in B$, let $\pi(w,\xi)$ be defined by $\pi(w,0) = w$ and 
\[
\pi(w, \xi) = T_{\varphi_{P,t}^{-1}}(w)  \text{ with } P = \frac{\xi}{|\xi|} \text{ and } t = (1 - |\xi|)^{-1} \text{ when } \xi \neq 0.
\]

\begin{lemma}\label{Lem:24VI20-L1}
The map $\pi: \mcS_0 \times B \mapsto C^{4,\alpha'}_+(\SSphere^n)$ is a $C^2$ diffeomorphism. 
\end{lemma}

\begin{proof}
The fact that $\pi$ is a bijection follows from \cite[Lemma 5.3]{Li95-JDE}. It is also clear that $\pi$ is $C^2$. Next, note that $u = \pi(w,\xi)$ if and only if 
\[
F(u,\xi) = \frac{1}{|\SSphere^n|}\int_{\SSphere^n} |u(x)|^{\frac{2n}{n-2}} \varphi_{P,t}^{-1}(x)\,dv_{\roundg} = 0,
\]
where $P = \frac{\xi}{|\xi|}$ and $t = (1 - |\xi|)^{-1}$. Thus, by the implicit function theorem, to show that $\pi^{-1}$ is $C^2$, it suffices to show that $\partial_\xi F(u_0,\xi_0)$ is an isomorphism for every $u_0 \in C^{4,\alpha'}_+(\SSphere^n)$ and $\xi_0 \in B$ such that $F(u_0,\xi_0) = 0$. In \cite[Lemma 5.4]{Li95-JDE}, this was done in Sobolev spaces. The same computation can be repeated verbatim giving the conclusion. We skip the details.
\end{proof}

Using Lemma \ref{Lem:24VI20-L1}, we will in the sequel `identify' an element $v \in C^{4,\alpha'}_+(\SSphere^n)$ with the pair $(w,\xi) = \pi^{-1}(v) \in \mcS_0 \times B$. As a consequence of Theorem \ref{thm9999} and Liouville-type theorem we have:

\begin{lemma}\label{Lem:25VI20-wConv}
Let $n\geq3$, $n/2\leq k\leq n$, and $0 < \alpha' < \alpha < 1$. Suppose that $K \in C^{2,\alpha}_+(\SSphere^n)$ satisfies the non-degeneracy condition \eqref{nondegeneracy}. If $v_{\mu_j} = \pi(w_{\mu_j}, \xi_{\mu_j})$ solves \eqref{Eq:27VII20-Xmu} for some sequence $\mu_j \rightarrow 0^+$, then $\xi_{\mu_j}$ stays in a compact subset of $B$ and 
\[
\lim_{j \rightarrow \infty} \|w_{\mu_j} - 1\|_{C^{4,\alpha'}(\SSphere^n)} = 0.
\]
\end{lemma}

\begin{proof} By Theorem \ref{thm9999}, first and second derivative estimates (see \cite{Chen05, GW03-IMRN}, \cite[Theorem 1.10]{Li09-CPAM}, \cite[Theorem 1.20]{LiLi03}, \cite{Wang06}) and Evans-Krylov's theorem, $\{v_{\mu_j}\}$ is bounded in $C^{4,\alpha}(\SSphere^n)$ and hence is relatively compact in $C^{4,\alpha'}(\SSphere^n)$. Since $\pi$ is a diffeomorphism, $\pi^{-1}$ maps compact sets into compact sets. It follows that $\{w_{\mu_j}\}$ is relatively compact in $\mcS_0$ and $\{\xi_{\mu_j}\}$ is bounded. Furthermore, if $w_*$ is a limit point of $\{w_{\mu_j}\}$, then $w_* \in \mcS_0$ and satisfies 
\[
\sigma_k(\lambda(A_{g_{w_*}})) = 2^{-k} \Big(\begin{array}{c}n\\k\end{array}\Big), \quad \lambda(A_{g_{w_*}}) \in \Gamma \quad \text{ on } \SSphere^n.
\]
By the Liouville-type theorem \cite[Theorem 3]{Viac00-Duke} (or \cite[Theorem 1.3]{LiLi05}), $w_* \equiv 1$. The conclusion follows.
\end{proof}

By a straightforward computation, for every fixed $\xi \in B$, the linearized operator of $F_\mu[\pi(\cdot,\xi)]$ at $\bar w \equiv 1$ is 
\[
\mcL := D_w (F_\mu \circ \pi)(w,\xi)]\Big|_{w = \bar w} = - d_{n,k}(\Delta + n) \quad \text{ with } \quad d_{n,k} := \frac{2^{2-k}}{n-2} \Big(\begin{array}{c}n\\k\end{array}\Big)
\]
and with domain $D(\mcL)$ being the tangent plane to $\mcS_0$ at $w = \bar w$:
\[
D(\mcL) := T_1(\mcS_0) = \Big\{\eta \in C^{4,\alpha'}(\SSphere^n): \int_{\SSphere^n} x \eta(x)\,dv_{\roundg}(x) = 0\Big\}.
\]
It is easy to check using the implicit function theorem that $\mcS_0$ is represented locally near $1$ as a graph over $T_1(\mcS_0)$: There is a twice differentiable map $\eta \in T_1(\mcS_0) \mapsto \zeta(\eta) \in \RR^{n+1}$ defined in a neighborhood of $0$ in $T_1(\mcS_0)$ with $\zeta(0) = 0$ and $D\zeta(0) = 0$ such that all $w \in \mcS_0$ sufficiently close to $1$ are of the form
\[
w(x) = 1 + \eta(x) + \zeta(\eta) \cdot x.
\]

It is well-known that $\mcL$ is an isomorphism from $D(\mcL)$ to
\[
R(\mcL) := \Big\{f \in C^{2,\alpha'}(\SSphere^n): \int_{\SSphere^n} x f(x)\,dv_{\roundg}(x) = 0\Big\}.
\]
Let $\Pi$ be a projection from $C^{2,\alpha'}(\SSphere^n)$ onto $R(\mcL)$ defined by
\[
\Pi f(x) = f(x) - \frac{n+1}{|\SSphere^n|} x \cdot \int_{\SSphere^n} y f(y)\,dv_{\roundg}(y).
\]

\begin{proposition}\label{prop233}
Let $n\geq3$, $n/2\leq k\leq n$, and $0 < \alpha' < \alpha < 1$. Suppose that $K \in C^{2,\alpha}(\SSphere^n)$ and let $F_\mu$ be defined by \eqref{Eq:25VI20-FmuDef}. Then for every $s_0 \in (0,1)$, there exists a constant $\mu_0 \in (0,1]$ and a neighborhood $\mcN$ of $1$ in $\mcS_0$  such that, for every $\mu \in (0,\mu_0]$ and $\xi \in \bar B_{s_0} \subset B$, there exists a unique $w_{\xi,\mu}\in \mcN$, depending smoothly on $(\xi,\mu)$, such that
\begin{equation}
\Pi(F_{\mu}[\pi(w_{\xi,\mu},\xi)])=0.\label{eq232}
\end{equation}
Furthermore, there exists some $C>0$ such that, for $\mu \in (0,\mu_0]$ and $|\xi|, |\xi'| \leq s_0$,
\begin{equation}
\|w_{\xi,\mu}-1\|_{C^{4,\alpha'}(\SSphere^{n})}\leq C\mu \Big\|K - 2^{-k}\binom{n}{k}\Big\|_{C^{2,\alpha}(\SSphere^{n})} ,
	\label{eq233}
\end{equation}
\begin{equation}
\|w_{\xi,\mu}-w_{\xi',\mu}\|_{C^{4,\alpha'}(\SSphere^{n})}\leq C\mu |\xi - \xi'| \Big\|K - 2^{-k}\binom{n}{k}\Big\|_{C^{2,\alpha}(\SSphere^{n})}.\label{Eq:01VII20-Lip}
\end{equation}

\end{proposition}

\begin{proof} We will only prove \eqref{Eq:01VII20-Lip}. The rest of the result follows from an immediate application of the implicit function theorem.

Fix $\xi,\xi' \in B_{s_0}$. Let $t = (1 - |\xi|)^{-1}$ and $t' = (1-|\xi'|)^{-1}$.  If both $\xi$ and $\xi'$ are non-zero, let $P = \frac{\xi}{|\xi|}$ and $P' = \frac{\xi'}{|\xi'|}$. If $\xi = 0$ but $\xi' \neq 0$, let $P = P ' = \frac{\xi'}{|\xi'|}$. If $\xi' = 0$ but $\xi \neq 0$, let $P = P ' = \frac{\xi}{|\xi|}$. If $\xi  = \xi' = 0$, take $P = P'$ to be any point on $\mathbb{S}^n$. Then
\[
F_\mu[\pi(w_{\xi,\mu},\xi')] = F_\mu[T_{\varphi_{P',t'}^{-1}}(w_{\xi,\mu})] = F_\mu [\pi(w_{\xi,\mu},0)] \circ \varphi_{P',t'}^{-1} 
\]
and so
\begin{align*}
F_\mu[\pi(w_{\xi,\mu},\xi')] - F_\mu[\pi(w_{\xi,\mu},\xi)]
	&= (F_0[\pi(w_{\xi,\mu},0)] - F_0[\pi(1,0)] ) \circ \varphi_{P',t'}^{-1}\\
		&\qquad - (F_0[\pi(w_{\xi,\mu},0)] - F_0[\pi(1,0)] ) \circ \varphi_{P,t}^{-1}.
\end{align*}
Thus, by \eqref{eq233}, 
\begin{align*}
\|F_\mu[\pi(w_{\xi,\mu},\xi')] - F_\mu[\pi(w_{\xi,\mu},\xi)]\|_{C^{2,\alpha'}(\SSphere^n)} \leq C\mu|\xi - \xi'|\Big\|K - 2^{-k}\binom{n}{k}\Big\|_{C^{2,\alpha}(\SSphere^{n})}.
\end{align*}
Estimate \eqref{Eq:01VII20-Lip} thus follows.
\end{proof}

Note that equation (\ref{eq232}) can be equivalently rewritten as
\begin{equation}
\sigma_{k}(\lambda(A_{g_{w_{\xi,\mu}}}))=K_{\mu}\circ\varphi_{P,t}(x)-\Lambda_{\xi,\mu}\cdot x\quad\mbox{on }\SSphere^{n},\label{0202-4}
\end{equation}
where $P = \frac{\xi}{|\xi|}$, $t = (1 - |\xi|)^{-1}$ and $\Lambda_{\xi,\mu} \in \RR^{n+1}$ is given by
\begin{equation}
\Lambda_{\xi,\mu}=-\frac{n+1}{|\SSphere^{n}|}\int_{\SSphere^{n}}F_{\mu}[\pi(w_{\xi,\mu},\xi)](x) x\,dv_{\roundg}(x).
	\label{Eq:25VI20-LambdaDef}
\end{equation}
It is clear that $\Lambda_{\xi,\mu}$ is smooth with respect to $(\xi,\mu)$. Furthermore, if $K$ satisfies the non-degeneracy condition \eqref{nondegeneracy}, then, by Lemma \ref{Lem:25VI20-wConv} and Proposition \ref{prop233}, for $\mu$ sufficiently close to $0$, $v_\mu$ solves \eqref{Eq:27VII20-Xmu} if and only if there exists some $\xi_\mu$ such that $v_\mu = \pi(w_{\xi_\mu,\mu},\xi_\mu)$ and $\Lambda_{\xi_\mu,\mu} = 0$.

We note that $\Lambda_{\xi,\mu}$ can also be computed more directly from the function $K$ as follows. In view of \eqref{0202-4} and the Kazdan-Warner identity (see \cite{Han06} and \cite{Viac00-AMS}), we have
\begin{equation*}
\int_{\SSphere^{n}}\langle \nabla(K_{\mu}\circ\varphi_{P,t} - \Lambda_{\xi,\mu}\cdot x),\nabla x_{i}\rangle w_{\xi,\mu}^{\frac{2n}{n-2}} dv_{\roundg}(x)=0,\quad 1\leq i\leq n+1.
\end{equation*}
It follows that, for $1\leq i\leq n+1$,
\begin{equation}
\frac{1}{\mu}\sum\limits_{j=1}^{n+1}\Lambda_{\xi,\mu}^{j}\int_{\SSphere^{n}}\langle \nabla x_{j},\nabla x_{i}\rangle w_{\xi,\mu}^{\frac{2n}{n-2}}dv_{\roundg}(x) = \int_{\SSphere^{n}}\langle \nabla(K\circ\varphi_{P,t}),\nabla x_i \rangle w_{\xi,\mu}^{\frac{2n}{n-2}}dv_{\roundg}(x).\label{0201-2}
\end{equation}

Note that, as $\mu \rightarrow 0$, by \eqref{eq233}, we have $w_{\xi,\mu} \rightarrow 1$ uniformly for $|\xi| \leq s_0$. This implies that the coefficient matrix on the left hand side of  \eqref{0201-2} is positive definite:
\begin{equation}
\left(\int_{\SSphere^{n}}\langle \nabla x_{j},\nabla x_{i}\rangle w_{\xi,\mu}^{\frac{2n}{n-2}}dv_{\roundg}(x) \right)_{1 \leq i,j \leq n+1} > 0.
	\label{Eq:24VII20-T1}
\end{equation}
This also implies that the right hand side of \eqref{0201-2} converges uniformly for $|\xi| \leq s_0$  to
\begin{equation}
G_i(\xi) := \int_{\SSphere^{n}} K\circ\varphi_{P,t}\, x_i \,dv_{\roundg}(x) \text{ where } P= \frac{\xi}{|\xi|} \text{ and } t = (1 - |\xi|)^{-1}.
	\label{Eq:25VI20-GDef}
\end{equation}

\begin{lemma}\label{Lem:13VII20}
Let $n\geq3$, $n/2\leq k\leq n$, and $\alpha \in (0,1)$. Suppose that $K \in C^{2,\alpha}_+(\SSphere^n)$ satisfies the non-degeneracy condition \eqref{nondegeneracy}. Let $\Lambda_{\xi,\mu}$ and $G$ be defined as in \eqref{Eq:25VI20-LambdaDef} and \eqref{Eq:25VI20-GDef}. Then there exist $\mu_0 \in (0,1]$ and $s_0 \in (0,1]$ such that, for all $\mu \in (0,\mu_0]$ and $s \in [s_0,1)$, the Brouwer degrees $\deg(\Lambda_{\xi,\mu},B_{s},0)$ and $\deg(G,B_{s},0)$are well-defined and
\[
\deg(\Lambda_{\xi,\mu},B_{s},0) = \deg(G,B_{s},0) = -(-1)^{n} + \deg(\nabla K, \textrm{Crit}_-(K)).
\]
\end{lemma}

\begin{proof}
In view of \eqref{0201-2}, \eqref{Eq:24VII20-T1}, the uniform convergence of the right hand side of \eqref{0201-2} to $G$ and the degree counting formula in \cite[Corollary 6.2]{Li95-JDE} (which gives $\deg(G,B_{s},0) = -(-1)^{n} + \deg(\nabla K, \textrm{Crit}_-(K))$), it suffices to show that $G$ does not have any zero near $\partial B$. 

Suppose by contradiction that there exist $\xi_i$ with $|\xi_i| \rightarrow 1$ such that $G(\xi_i) = 0$. Let $P_i = \frac{\xi_i}{|\xi_i|}$, $t_i = (1 - |\xi_i|)^{-1}$ and assume, without loss of generality, that $P_i \rightarrow P$. 

Let $\Phi: \RR^n \rightarrow \SSphere^n $  be the stereographic projection with $P$ being the south pole to the equatorial plane of $\SSphere^{n}$, and define $\tilde K = K \circ \Phi$. The equations
\[
0 = G(\xi_i) = \int_{\SSphere^{n}} K\circ\varphi_{P_i,t_i}\, x \,dv_{\roundg}(x) = - \int_{\SSphere^{n}}\langle \nabla(K\circ\varphi_{P_i,t_i}),\nabla x_j \rangle  dv_{\roundg}(x)
\]
then transform into
\[
0 = \int_{\RR^n} \nabla \tilde K U_i^{\frac{2n}{n-2}} dy \quad \text{ and } \quad  0 = \int_{\RR^n} y \cdot \nabla \tilde K U_i^{\frac{2n}{n-2}}dy
\]
where
\[
U_i = \Big(\frac{\lambda_i}{1 + \lambda_i^2 |y - y_i|^2}\Big)^{\frac{n-2}{2}}
\]
for some $y_i \in \RR^n$, $\lambda_i > 0$ such that $y_i \rightarrow 0$ and $\lambda_i \rightarrow \infty$. As shown in the proof of Theorem \ref{thm9999} (see the argument following \eqref{rushan} and \eqref{rushan3}), this implies that $\nabla_{\roundg} K(P) = 0$ and $\Delta_{\roundg} K(P) = 0$, contradicting \eqref{nondegeneracy}.
\end{proof}

The last piece for the proof of Theorem \ref{main'} is the following result.

\begin{theorem}\label{Thm:dff}
Let $n\geq3$, $n/2\leq k\leq n$, and $0 < \alpha < 1$. Suppose that $K \in C^{2,\alpha}_+(\SSphere^n)$ satisfies the non-degeneracy condition \eqref{nondegeneracy}. Let $F_1$ and $\mcO$ be defined as in \eqref{Eq:25VI20-FmuDef} and \eqref{Eq:25VI20-ODef} with $\alpha' = \alpha$. Then
\[
\deg(F_1, \mcO, 0)=  -1 + (-1)^{n}\deg(\nabla K, \textrm{Crit}_-(K)).
\]

\end{theorem}

\begin{proof} Fix some $\alpha' \in (0,\alpha)$. By Lemma \ref{Lem:13VII20}, it suffices to show that there exist $\mu_0 \in (0,1]$ and $s_0 \in (0,1)$ such that
\[
\deg(F_\mu, \mcO, 0) =  (-1)^{n}\deg(\Lambda_{\xi,\mu},B_{s},0) \text{ for all } \mu \in (0,\mu_0], s \in (s_0,1),
\]
where $\mcO$ is now as in \eqref{Eq:25VI20-FmuDef} with $\alpha' < \alpha$.

Note that, by Lemma \ref{Lem:25VI20-wConv} and Proposition \ref{prop233}, there exist $s_0 \in (0,1)$, an open neighborhood $\mcN$ of $1$ in $\mcS_0$ and a small $\mu_0 > 0$ such that all solutions of $F_\mu[v] = 0$ with $0 < \mu \leq \mu_0$ belong to $\pi(\mcN \times B_{s_0})$ and have the form $v = \pi(w_{\xi,\mu},\xi)$ for some $\xi \in B_{s_0}$ satisfying $\Lambda_{\xi,\mu} = 0$. By the excision property, we have
\[
\deg (F_\mu,\mcO,0) = \deg (F_\mu,\pi(\mcN \times B_{s_0}),0) \text{ for all } 0 < \mu \leq \mu_0.
\]

Consider first the case that $F_\mu$ has only non-degenerate zeroes in $\pi(\mcN \times B_{s_0})$ which correspond to non-degenerate zeroes of $\Lambda_{\xi,\mu}$. We then have
\[
\deg (F_\mu,\pi(\mcN \times B_{s_0}),0) = \sum_{\xi \in B_{s_0}: \Lambda_{\xi,\mu} = 0} \deg(F_\mu, \pi(w_{\xi,\mu},\xi)),
\]
where $\deg(F_\mu, \pi(w_{\xi,\mu},\xi))$ is the local degree of $F_\mu$ at $\pi(w_{\xi,\mu},\xi)$. 
Therefore, to conclude the proof in the present case, it suffices to show that, if $\xi_0$ is a zero of $\Lambda_{\xi,\mu}$, then $DF_\mu[\pi(w_{\xi_0,\mu},\xi_0)]$ is non-degenerate and 
\begin{equation}
\deg(F_\mu, \pi(w_{\xi_0,\mu},\xi_0)) = -\deg(-\Lambda_{\xi,\mu}, \xi_0).
	\label{Eq:30VI20-Deg6}
\end{equation}

As explained earlier, for $w$ in a neighborhood of $1$ in $\mcS_0$, we can uniquely write $w = 1 + \eta + \zeta(\eta) \cdot x$ where $ \eta \in T_1(\mcS_0)$ and $\zeta(\eta) \in \RR^{n+1}$. Let $w_{\xi,\mu} = 1 + \eta_{\xi,\mu} + \zeta_{\xi,\mu} \cdot x$. 

Define
\[
\tilde F_\mu(\eta,\xi) = F_\mu[\pi(1 + \eta + \zeta(\eta) \cdot x,\xi)].
\]

Recall that $D_\eta \tilde F_\mu(1,\xi) = \mcL = -d_{n,k} (\Delta + n)$ is a single simple negative eigenvalue on $D(\mcL) = T_1(\mcS_0)$. Thus, by Lemma \ref{Lem:25VI20-wConv} (and Proposition \ref{prop233}), when $\mu_0$ is sufficiently small, $D_\eta \tilde F_\mu(\eta_{\xi,\mu},\xi)$ also has a single simple negative eigenvalue.

Next, we have that $0 = \Pi \tilde F_\mu(\eta_{\xi,\mu},\xi) = \tilde F_\mu(\eta_{\xi,\mu},\xi) + \Lambda_{\xi,\mu} \cdot x$. Hence
\[
0 =   D_\xi \tilde F_\mu(\eta_{\xi,\mu},\xi) + D_\xi (\Lambda_{\xi,\mu} \cdot x) + D_\eta \tilde F_\mu(\eta_{\xi,\mu},\xi) (D_\xi \eta_{\xi,\mu}).
\]
By \eqref{Eq:01VII20-Lip}, we have $\|D_\xi \eta_{\xi,\mu}\|_{C^{4,\alpha'}(\SSphere^n)} \leq C\mu$. Hence as $\mcL(D_\xi \eta_{\xi,\mu}) = 0$, we have by \eqref{eq233} that
\[
\|D_\eta \tilde F_\mu(\eta_{\xi,\mu},\xi) (D_\xi \eta_{\xi,\mu})\|_{C^{2,\alpha'}(\SSphere^n)} \leq C\mu^2.
\]
Recalling \eqref{0201-2}, \eqref{Eq:25VI20-GDef} and the fact that $G$ has non-degenerate zeros, we deduce that $DF_\mu[\pi(w_{\xi_0,\mu},\xi_0)]$ is non-degenerate and its number of negative eigenvectors is equal to that of $-D_\xi \Lambda_{\xi_0,\mu}$. This proves \eqref{Eq:30VI20-Deg6} and hence concludes the proof in the case $F_\mu$ has only non-degenerate zeroes in $\pi(\mcN \times B_{s_0})$.

Consider now the general case. For $a \in \RR^{n+1}$, let $K^{(a)}(x) = K(x) + a\cdot x$, $K_\mu^{(a)}(x) = K_\mu(x) + a\cdot x$ and define $F_\mu^{(a)}[v] = \sigma_k(\lambda(A_{g_v})) - K_\mu^{(a)} = F_\mu[v] - \mu a\cdot x$. Using again Theorem \ref{thm9999} we have that $\deg(F_\mu^{(a)}, \mcO, 0)$ is well-defined and is independent of $a$ when $|a|$ is sufficiently small. Now observe that $\Pi(F_{\mu}^{(a)}[v])=\Pi(F_{\mu}[v])$, we have that the $w^{(a)}_{\xi,\mu}$ corresponding to $K^{(a)}$ obtained in Proposition \ref{prop233} is in fact $w^{(a)}_{\xi,\mu} = w_{\xi,\mu}$. It thus follows that the function $\Lambda^{(a)}_{\xi,\mu}$ corresponding to $K^{(a)}$ is given by
\[
\Lambda_{\xi,\mu}^{(a)} = -\frac{n+1}{|\SSphere^{n}|}\int_{\SSphere^{n}}F_{\mu}^{(a)}[\pi(w_{\xi,\mu},\xi)](x) x\,dv_{\roundg}(x) = \Lambda_{\xi,\mu} + \mu a.
\]
By Sard's theorem, we thus have for almost all $a$ that $\Lambda_{\xi,\mu}^{(a)}$ and so $F_\mu^{(a)}$ have only non-degenerate zeros. By the previous case, we have $\deg(F_{\mu}^{(a)}, \mcO, 0)= (-1)^{n} \deg(\Lambda_{\xi,\mu} + \mu a,B_s,0)$ for those $a$. The conclusion thus follows from the continuity property of the degree.
\end{proof}

\begin{proof}[Proof of Theorem \ref{main'}]
Estimate (\ref{uniform estimate}) is proved by Theorem \ref{thm9999}. Let us assume that $\deg(\nabla K,\textrm{Crit}_-(K)) \neq (-1)^n$ and prove the existence of a solution to \eqref{yq}. Let $K_j$ be a sequence of functions in $C^{2,\alpha}_+(\SSphere^n)$ which converges to $K$ in $C^2$. Then, for all sufficiently large $j$, $K_j$ satisfies the non-degeneracy condition \eqref{nondegeneracy} and $\deg(\nabla K_i,\textrm{Crit}_-(K_i)) \neq (-1)^n$. By Theorem \ref{Thm:dff}, there exists $v_j \in C^{4,\alpha}_+(\SSphere^n)$ such that 
\[
\sigma_k(\lambda(A_{g_{v_j}})) = K_j , \qquad \lambda(A_{g_{v_j}}) \in \Gamma_k \text{ on } \SSphere^n.
\]
By Theorem \ref{thm9999}, we have
\[
\|\ln v_j\|_{C(\SSphere^n)}  \leq C,
\]
and so, by first and second derivative estimates for the $\sigma_k$-Yamabe equation and Evans-Krylov's theorem,
\[
\|\ln v_j\|_{C^{2,\alpha}(\SSphere^n)}  \leq C.
\]
Sending $j\rightarrow \infty$ we arrive at the conclusion.
\end{proof}

\appendix
\section{Miscellaneous results on the conformal Hessian}\label{App:Misc}
We collect some results which can be used to give another proof of Lemma \ref{Lem:20VI20-KeyNew}.

\subsection{Convexity of the conformal Hessian}

Let $\mathring{g}$ denote the Euclidean metric on $\RR^n$. Then the $(1,1)$-Schouten tensor of the metric $g =  w^{-2} \mathring{g}$ is given by $w A_w$ with
\[
A_w := \nabla^2 w - \frac{1}{2w} |\nabla w|^2\,I.
\]

\begin{lemma}
\label{Lem:23VI20-L1}
Suppose that $0 < w_1, w_2 \in C^2(\Omega)$. Then
\[
A_{\frac{1}{2}(w_1 + w_2)} \geq \frac{1}{2} (A_{w_1} + A_{w_2}) \text{ in } \Omega.
\]
\end{lemma}

\begin{proof}
Let $\bar w = \frac{1}{2}(w_1 + w_2)$. We have 
\begin{align*}
A_{\bar w} &= \frac{1}{2} (\nabla^2 w_1 + \nabla^2 w_2) -  \frac{1}{8\bar w}\big|\nabla w_1 + \nabla w_2\big|^2 I\\
	 &= \frac{1}{2} (A_{w_1} + A_{w_2}) 
	  	+ \frac{1}{8\bar w}
			\Big\{(w_1 + w_2) \Big(\frac{1}{w_1} |\nabla w_1|^2 + \frac{1}{w_2}|\nabla w_2|^2\Big)
				 - \big|\nabla w_1 + \nabla w_2\big|^2 \Big\}I.
\end{align*}
As the term in the curly braces is non-negative thanks to Cauchy-Schwarz' inequality, the conclusion follows.
\end{proof}

\subsection{A monotonicity estimate}

\begin{lemma}\label{lihaotian}
Let $n\geq 3$ and $n/2<k\leq n$. Suppose $u$ is a $C^{2}$ positive radially symmetric function satisfying $\lambda(A^{u})\in\bar{\Gamma}_{k}$ in $B_{2}$. Suppose further that $r^{\frac{n-2}{2}}u(r)$ is non-increasing on $[r_{1},r_{2}]\subset(0,2)$. Then
\begin{equation}\label{yuxuanyang}
1\leq\frac{r_{2}^{n-2}u(r_{2})}{r_{1}^{n-2}u(r_{1})}\leq 2^{\frac{(n-2)k}{2k-n}}.
\end{equation}
\end{lemma}

\begin{proof} 
As in \cite{ChangHanYang05-JDE, Viac00-Duke}, we consider the differential inclusion $\lambda(A^u) \in \bar\Gamma_k$ using cylindrical metric. Let $t = \ln r$, $t_1 = \ln r_1$, $t_2 = \ln r_2$, and $\xi(t) = -\frac{2}{n-2}\ln u(r) - \ln r$ so that $u^{\frac{4}{n-2}}\mathring{g} = e^{-2\xi} (dt^2 + g_{\mathbb{S}^{n-1}})$. A direct computation gives
\[
\sigma_\ell(\lambda(A^u))
	= \frac{1}{2^{\ell-1}}\,\Big(\begin{array}{c} n-1\\ \ell-1\end{array}\Big)\,e^{2\ell\xi}(1 - |\xi'|^2)^{\ell-1}[\xi'' + \frac{n-2\ell}{2\ell}(1-|\xi'|^2)],
\]
where here and below $'$ denotes differentiation with respect to $t$. As $k \geq 2$, we have $\sigma_\ell(\lambda(A^u)) \geq 0$ for $\ell = 1,2$ and so $1 - |\xi'|^2 > 0$. This implies that $r^{n-2} u(r)$ is non-decreasing (and $u(r)$ is non-increasing). This proves the left half of (\ref{yuxuanyang}). (See also \cite[Lemma 2.7]{LiNgBocher}.)

As $r^{\frac{n-2}{2}} u(r)$ is non-increasing, $\xi' \geq 0$. Hence, as $\sigma_k(\lambda(A^u))\geq 0$,
\[
H(\xi,\xi') := e^{(2k-n)\xi}(1 - |\xi'|^2)^k \text{ is non-increasing.}
\]
It follows that 
\[
e^{\frac{(2k-n)\xi}{k}}(1 - |\xi'|^2) \leq H(\xi(t_1), \xi'(t_1))^{\frac{1}{k}} \leq e^{\frac{(2k-n)\xi(t_1)}{k}}.
\]
Letting $e^{\frac{(2k-n)(\xi - \xi(t_1))}{2k}} = \cosh \eta$, we then get $\eta' \geq \frac{2k - n}{2k}$. This gives
\begin{align*}
\xi(t_2) - \xi(t_1) 
	&=\frac{2k}{2k-n} \log\cosh \eta(t_2) \geq \frac{2k}{2k-n} \log\cosh[\frac{2k - n}{2k}(t_2 - t_1) ] \\
	&\geq t_2 - t_1 - \frac{2k}{2k-n}\ln 2.
\end{align*}
The right half of (\ref{yuxuanyang}) follows.
\end{proof}

To illustrate, let us apply the above result to give a proof of \eqref{zhongyao} when $k > n/2$. By \eqref{as}, we have 
\begin{equation}\label{ad}
\mbox{$r^{\frac{n-2}{2}}\hat{u}_{i}(r)$ has a unique critical (maximum) point in $(0,R_i/\lambda_i)$},
\end{equation}
where
\[
\hat{u}_{i}(r)
	=\left(\frac{1}{|\partial B_{r}(y_{i})|}\int_{\partial B_{r}(y_{i})}u_{i}^{-\frac{2}{n-2}}\right)^{-\frac{n-2}{2}}.
\]
It is a fact that \eqref{ad} can be improved to
\begin{equation}\label{adX}
\mbox{$r^{\frac{n-2}{2}}\hat{u}_{i}(r)$ has a unique critical (maximum) point in $(0,\rho)$}
\end{equation}
for some $\rho > 0$ independent of $i$. (We decide to skip the proof of this for brevity.)

\begin{proof}[An alternative proof of Lemma \ref{Lem:20VI20-KeyNew} when $k > n/2$ and when \eqref{adX} holds]
 By (\ref{genqian}), the definition of $\hat{u}_{i}$ and the convexity of the operator $A_{w}$ (see Lemma \ref{Lem:23VI20-L1}), $\hat{u}_{i}$ is a $C^{2}$ positive radially symmetric function satisfying $
\lambda(A^{\hat{u}_{i}})\in\bar{\Gamma}_{k}$ on $\RR^{n}$. 
By (\ref{ad}) and \eqref{adX}, we have that $r^{\frac{n-2}{2}}\hat{u}_{i}(r)$ is non-increasing on $[R_i/\lambda_i,\rho]$. Without loss of generality, we may assume that $\rho<1$. It follows from Lemma \ref{lihaotian} and (\ref{as}) that 
\begin{equation*}
r^{n-2}\hat{u}_{i}(r)\leq CR_i^{n-2}\lambda_i^{-(n-2)}\hat{u}_{i}(R_{i}/\lambda_i)\leq Cu_{i}^{-1}(0),~~\forall~R_{i}/\lambda_i \leq r\leq\rho.
\end{equation*}
Estimate (\ref{zhongyao}) then follows in view of Harnack estimates (cf. \eqref{Eq:20VI20-H2}).
\end{proof}

\newcommand{\noopsort}[1]{}

\end{document}